\pdfoutput=1
\documentclass[twoside,reqno, table]{elsarticle}
\usepackage[utf8]{inputenc}
\usepackage{xcolor}
\usepackage{pdfcolmk}
\usepackage{prettyref}
\usepackage{mathtools}
\usepackage{amsmath}
\usepackage{amsthm}
\usepackage{amssymb}
\usepackage{stmaryrd}
\PassOptionsToPackage{normalem}{ulem}
\usepackage{ulem}
\usepackage{microtype}
\usepackage[unicode=true,
 bookmarks=false,
 breaklinks=false,pdfborder={0 0 1},backref=false,colorlinks=true]
 {hyperref}
\hypersetup{
 pdfstartview=FitV,linkcolor=black,citecolor=black}

\makeatletter

\providecommand{\tabularnewline}{\\}
\providecolor{lyxadded}{rgb}{0,0,1}
\providecolor{lyxdeleted}{rgb}{1,0,0}

\DeclareRobustCommand{\lyxsout}[1]{\ifx\\#1\else\sout{#1}\fi}
 
\numberwithin{equation}{section}
\numberwithin{figure}{section}
\theoremstyle{plain}
\newtheorem{thm}{\protect\theoremname}[section]
\theoremstyle{plain}
\newtheorem{cor}[thm]{\protect\corollaryname}
\theoremstyle{definition}
\newtheorem{defn}[thm]{\protect\definitionname}
\theoremstyle{remark}
\newtheorem{rem}[thm]{\protect\remarkname}
\theoremstyle{plain}
\newtheorem{lem}[thm]{\protect\lemmaname}
\theoremstyle{plain}
\newtheorem{prop}[thm]{\protect\propositionname}
\theoremstyle{remark}
\newtheorem{examp}[thm]{Example}
 
\journal{~}
\usepackage{float}
\usepackage{bbm}
\usepackage{txfonts}
\usepackage{textcomp}
 
\usepackage{dsfont}
\usepackage{mathrsfs}
\usepackage{graphicx} 
\usepackage{caption}
\usepackage{subcaption} 
\usepackage{tikz}
\usepackage{lineno}
\usepackage{pgfplots}
\pgfplotsset{compat=1.15}
\usepackage{mathrsfs}
\usetikzlibrary{arrows} 
\usepackage{enumitem}

\renewcommand{\rho}{\varrho}

\newcommand{\idvtex}{\mathrm{id}}

\DeclareMathOperator{\dom}{dom}
\DeclareMathOperator{\diam}{diam}

\DeclareMathOperator{\supp}{supp}
\DeclareMathOperator{\card}{card}
\DeclareMathOperator{\var}{var}

\renewcommand{\tilde}{\widetilde}

\renewcommand{\phi}{\varphi}

\def\reftext{}

\newrefformat{prop}{Proposition~\ref{#1}}
\newrefformat{fact}{Fact~\ref{#1}}
\newrefformat{exa}{Example~\ref{#1}}
\newrefformat{cor}{Corollary~\ref{#1}}
\newrefformat{assu}{Assumption~\ref{#1}}
\newrefformat{subsec}{Section~\ref{#1}}
\newrefformat{sec}{Section~\ref{#1}}
\newrefformat{def}{Definition~\ref{#1}}
\newrefformat{eq}{(\ref{#1})}

\definecolor{lightgray}{gray}{0.9}

\makeatother

\providecommand{\corollaryname}{Corollary}
\providecommand{\definitionname}{Definition}
\providecommand{\lemmaname}{Lemma}
\providecommand{\propositionname}{Proposition}
\providecommand{\remarkname}{Remark}
\providecommand{\theoremname}{Theorem}

\begin{document}

\begin{frontmatter}{}

\title{Spectral asymptotics of Kre\u{\i}n-Feller operators for weak Gibbs
measures on self-conformal fractals with overlaps}

\author{Marc~Kesseböhmer\fnref{fn1}\corref{cor1}}

\ead{mhk@uni-bremen.de}

\author{Aljoscha~Niemann\corref{cor1}}

\ead{niemann1@uni-bremen.de}

\address{Fachbereich 3 -- Mathematik und Informatik, University of Bremen,
Bibliothekstr. 5, 28359 Bremen, Germany}

\cortext[cor1]{Corresponding authors}
\begin{abstract}
We study the spectral dimensions and spectral asymptotics of Kre\u{\i}n-Feller
operators for weak Gibbs measures on self-conformal fractals with
or without overlaps. We show that, restricted to the unit interval,
the $L^{q}$-spectrum for every weak Gibbs measure $\rho$ with respect
to a $\mathcal{C}^{1}$-IFS exists as a limit. Building on recent
results of the authors, we can deduce that the spectral dimension
with respect to a weak Gibbs measure exists and equals the fixed point
of its $L^{q}$-spectrum. For an IFS satisfying the open set condition,
it turns out that the spectral dimension equals the unique zero of
the associated pressure function. Moreover, for a Gibbs measure with
respect to a $\mathcal{C}^{1+\gamma}$-IFS under OSC, we are able
to determine the asymptotics of the eigenvalue counting function.
\end{abstract}

\end{frontmatter}{}

\tableofcontents{}
\section{Introduction and statement of main results}
\label{sec1}

We investigate the spectral properties under Dirichlet boundary conditions
of the classical Kre\u{\i}n--Feller operator $\Delta _{\varrho }$ for weak
Gibbs measures $\varrho $ with respect to a conformal iterated function
systems on $\left [0,1\right ]$ with or without overlaps (see Section~\ref{sec:Conformal-iterated-function}).
Spectral properties of the operator $\Delta _{\varrho }$ have attracted
much attention in the last century, beginning with Feller
\citep{Fe57}, Kac \citep{MR0107037}, Hong and Uno \citep{MR118891}, McKean
and Ray \citep{MR146444}, Kotani and Watanabe \citep{MR661628}, Fujita
\citep{Fu87}, Solomyak and Verbitsky \citep{MR1328700} and more recently
by Vladimirov and She\u{\i}pak \citep{MR3185121}, Faggionato
\citep{MR2892328}, Arzt \citep{Arzt_diss,A15b}, Ngai
\citep{MR2828537}, Ngai, Tang and Xie \citep{MR3809018,MR4176086},  Freiberg,
Minorics \citep{MR4048458,Freiberg:aa}, and by the authors in
\citep{KN2022,KN2022b,KN21b}.

In the framework of the weak approach starting from the Dirichlet form
$\mathcal{E}_{\varrho }$ as defined in Section~\ref{sec:Dirichlet-forms-for}
it is well known that there exists an orthonormal system of eigenfunctions
of $\Delta _{\varrho }$ with non-negative eigenvalues
$\left (\lambda _{\varrho }^{n}\right )_{n\in \mathbb{N}}$ in increasing
order tending to $\infty $ whenever the support $\supp (\varrho )$ of
$\varrho $ is not a finite set. We denote the number of eigenvalues of
$\Delta _{\varrho }$ not exceeding $x\geq 0$ by
$N_{\varrho }\left (x\right )$ and refer to $N_{\varrho }$ as the
\emph{eigenvalue counting function}. We define the upper and lower exponent
of divergence by
\begin{equation*}
\underline{s}_{\varrho }\coloneqq \liminf _{x\rightarrow \infty }
\frac{\log \left (N_{\varrho }(x)\right )}{\log (x)}\quad \text{and }
\;\overline{s}_{\varrho }\coloneqq \limsup _{x\to \infty }
\frac{\log \left (N_{\varrho }(x)\right )}{\log (x)}
\end{equation*}
and refer to these numbers as the \emph{upper}, resp. \emph{lower},
\emph{spectral} \emph{dimension}. If the two values coincide we denote the
common value by $s_{\varrho }$ and call it the
\emph{spectral dimension} of $\Delta _{\varrho }$, or of
$\mathcal{E}_{\varrho }$, respectively. Note that, we always have
$\overline{s}_{\varrho }\leq 1/2$, which has been shown in
\citep{MR0209733,MR0217487}. If $\varrho $ has an non-trivial absolutely
continuous part $\sigma \Lambda |_{[0,1]}$ and a singular part
$\eta $ with $\eta \left (\left \{  0,1\right \}  \right )=0$, then we have
\begin{equation*}
\lim _{x\to \infty }
\frac{N_{\eta +\sigma \Lambda |_{[0,1]}}(x)}{x^{1/2}}=\frac{1}{\pi }
\int _{[0,1]}\sqrt{\sigma }\;\mathrm{d}\Lambda ,
\end{equation*}
where $\Lambda $ denotes the Lebesgue measure on $\mathbb{R}$ (see
\citep{MR146444,MR0278126}). In particular, the spectral dimension equals
$s_{\eta +\sigma \Lambda }=1/2$. A first account for smooth densities and
no singular part is contained in the famous work \citep{Weyl:1912} of Weyl.

In the case of self-similar measures $\varrho $ under the open set condition
(OSC) with contraction rates
$r_{1},\dots ,r_{n}\in \left (-1,1\right )$ and probability weights
$p_{1},\dots ,p_{n}\in \left (0,1\right )$, $n\geq 2$, it has been shown
in \citep{MR118891,Fu87,MR1298682} that the spectral dimension
$s_{\varrho }$ is given by the unique $q>0$ such that
\begin{equation}
\sum _{i=1}^{n}\left (p_{i}\left |r_{i}\right |\right )^{q}=1.
\label{eq:SpectralDimensionSelfSimiliar}
\end{equation}
We will generalize this result in three ways.
\begin{itemize}
\item We provide a first contribution to the nonlinear setting in a broad
sense. More specifically, we consider weak Gibbs measures on fractals which
are generated by non-trivial  $\mathcal{C}^{1}$ iterated function systems ($\mathcal{C}^{1}$-IFS) under the OSC. It turns
out that the spectral dimension is given by the zero of the associated
pressure function (see \reftext{(\ref{eq:Pressure})}) which is a natural generalization
of \reftext{(\ref{eq:SpectralDimensionSelfSimiliar})}.
\item As a second novelty, we drop the assumption of the OSC and allow
overlaps. In this situation the computation of the spectral dimension is
much more involved compared to \reftext{(\ref{eq:SpectralDimensionSelfSimiliar})}.
However, building on ideas developed in \citep{KN2022} combined with results
from \citep{MR1838304,MR2322179}, we are able to specify the spectral dimension
as the fixed point of the associated $L^{q}$-spectrum as defined below
in \reftext{(\ref{eq:Lq-spectrum})}. Finally, using a recent result of
\citep{Barral2020}, for the self-similar case with possible overlaps and
some additional assumption, we can express the spectral dimension in terms
of $\tau $ which is implicitly given by
$\sum _{i=1}^{n}p_{i}^{q}\left |r_{i}\right |^{\tau (q)}=1$.
\item Our final contribution to the nonlinear setting concerns Gibbs measures
on fractals generated by $\mathcal{C}^{1+\gamma }$-IFS's under the OSC. For
this class, we are able to prove spectral asymptotics using renewal theory
developed for a dynamical context, as in \citep{MR3881115,MR3600649}.
\end{itemize}
Investigating the spectral dimension of $\Delta _{\varrho }$, the authors
have recently shown in \citep{KN2022} that the
\emph{$L^{q}$-spectrum} $\beta _{\varrho }$ of $\varrho $ carries the crucial
information. For $q\geq 0$, it is given by
\begin{equation}
\beta _{\varrho }\left (q\right )\coloneqq \limsup _{n\rightarrow
\infty }\beta _{n}^{\varrho }\left (q\right )\quad \text{with\quad \ }
\beta _{n}^{\varrho }\left (q\right )\coloneqq \frac{1}{\log 2^{n}}
\log \sum _{C\in \mathcal{D}_{n}}\varrho \left (C\right )^{q},
\label{eq:Lq-spectrum}
\end{equation}
where
$\mathcal{D}_{n}\coloneqq \left \{  A_{n}^{k}:k\in \mathbb{Z},\,
\varrho \left (A_{n}^{k}\right )>0\right \}  $ and
$A_{n}^{k}\coloneqq \left (\left (k-1\right )2^{-n},k2^{-n}\right ]$. Each
$\beta _{n}^{\varrho }$ defines a non-increasing, differentiable and convex
function with unique fixed point $q_{n}^{\varrho }\in (0,1)$, i.e.
$\beta _{n}^{\varrho }\left (q_{n}^{\varrho }\right )=q_{n}^{\varrho }$.
We have
$\beta _{n}^{\varrho }\left (1\right )=\beta _{\varrho }\left (1
\right )=0$, $n\in \mathbb{N}$, and
$\beta _{\varrho }\left (0\right )$ is equal to the
\emph{upper Minkowski dimension}
$\overline{\dim }_{M}\left (\supp \left (\varrho \right )\right )$ of the
support $\supp \left (\varrho \right )$ of $\varrho $. The following quantity
\begin{equation*}
q_{\varrho }\coloneqq \limsup _{n\rightarrow \infty }q_{n}^{\varrho }
\end{equation*}
has been introduced by the authors in \citep{KN2022} and plays a central
role for the spectral problem. In fact, by extending and combining ideas
from there and \citep{Barral2020,MR1838304}, we can prove that for weak
Gibbs measures $\varrho $, as defined in Section~\ref{sec:Conformal-iterated-function},
the spectral dimension $s_{\varrho |_{\left (0,1\right )}}$ always exists
and is equal to $q_{\varrho }$, which generalizes previous results for
linear IFS under the OSC in \citep{MR118891,MR1328700}. The restriction of
$\varrho $ to the open unit interval guarantees that there are no atoms
at the boundary points, which on the one hand allows the weak Dirichlet
approach, while on the other hand the $L^{q}$-spectrum on $[0,1]$ and the
value of $q_{\varrho }$ are not affected by this restriction.
\begin{thm} 
\label{thm:WeakGibbs}%
Let $\varrho $ be a weak Gibbs measure on $[0,1]$ with respect to a non-trivial
$\mathcal{C}^{1}$-IFS (with or without overlap). Then the spectral dimension
$s_{\varrho |_{\left (0,1\right )}}$ exists and equals
$q_{\varrho }$. If, additionally, the OSC is fulfilled, then
$q_{\varrho }$ coincides with the unique zero $z_{\varrho }$ of the pressure
function as defined in \reftext{(\ref{eq:Pressure})}.
\end{thm}
 
\begin{cor}
\label{cor1.2}
Let $\varrho $ be a weak Gibbs measure with respect to a
$\mathcal{C}^{1}$-IFS (with or without overlap). If
$q_{\varrho }<1/2$, then $\varrho $ is singular with respect to
$\Lambda $.
\end{cor}

As a by-product, using ideas of Riedi \citep{MR1312056,Riedi_diss}, we
can show that the Minkowski dimension of the self-conformal set generated
by a $\mathcal{C}^{1}$-IFS with overlaps always exists (\reftext{Proposition~\ref{prop:_ApplicationRiedi}})---a
fact we could not find in the literature.

In the special case of dimensionally regular linear IFS (cf. \reftext{Definition~\ref{def:DimesionRegular}})
we can apply a recent result by Barral and Feng \citep{Barral2020} to compute
the spectral dimension more explicitly. We consider an IFS given by contracting
similarities
$\Phi =\left (T_{i}:[0,1]\rightarrow [0,1]:i=1,\dots ,n\right )$ where
$T_{i}(x)=r_{i}x+b_{i}$, $x\in \mathbb{R}$ with
$b_{i}\in \mathbb{R}$, $\left |r_{i}\right |<1$, $i=1,\dots ,n$. For a
given probability vector $(p_{1},\ldots ,p_{n})\in (0,1)^{n}$, we call
the unique Borel probability measure $\varrho $ satisfying
\begin{equation}
\varrho (A)=\sum _{i=1}^{n}p_{i}\cdot \varrho \circ T_{i}^{-1}(A),\:A
\in \mathfrak{B}([0,1])
\label{eq:Simi-1}
\end{equation}
the \emph{self-similar measure} of $\Phi $ with probability vector
$(p_{1},\ldots ,p_{n})$, where $\mathfrak{B}([0,1])$ denotes the Borel
$\sigma $-algebra of $[0,1]$. For every $q\in \mathbb{R}$ let
$\tau (q)$ be the unique solution of
\begin{equation}
\sum _{i=1}^{n}p_{i}^{q}\left |r_{i}\right |^{\tau (q)}=1,
\label{eq:Def_tau}
\end{equation}
which defines an analytic function $q\mapsto \tau (q)$. With the help of
$\tau $, the
\emph{similarity dimension of its attractor $\supp \varrho $ }is set to
be
\begin{equation*}
\dim _{S}(\supp \varrho )\coloneqq \tau (0)
\end{equation*}
and we define the
\emph{similarity dimension of the measure $\varrho $ }to be
\begin{equation*}
\dim _{S}(\varrho )\coloneqq -\tau '(1)=
\frac{\sum _{i=1}^{n}\log \left (p_{i}\right )p_{i}}{\sum _{i=1}^{n}\log \left (\left |r_{i}\right |\right )p_{i}}.
\end{equation*}
 
\begin{defn} 
\label{def:DimesionRegular}
An IFS $\Phi \coloneqq \left (\varphi _{i}\right )_{i=1,\ldots ,n}$ of
similarities on $\mathbb{R}$ with contraction rates
$\left (r_{1},\ldots ,r_{n}\right )$ is said to be
\emph{dimensionally regular,} if every self-similar measure
$\varrho $ of $\Phi $ with probability vector
$(p_{1},\ldots ,p_{n})\in \left (0,1\right )^{n}$ has Hausdorff dimension
\begin{equation*}
\dim _{H}(\varrho )=\min \left \{  1,\dim _{S}(\varrho )\right \}  .
\end{equation*}
\end{defn} 
\begin{rem}
\label{rem1.4}
From Hochman \citep[Theorem 1.1]{MR3224722}, it follows that if the similarities
$\left (T_{i}\right )_{i=1,\ldots ,n}$ satisfy the
\emph{exponential separation condition} (ESC) (see e.g.
\citep[Definition 2.2]{Barral2020}), then $\varrho $ is dimensionally regular.
\end{rem}
 
\begin{thm} 
\label{thm:SpectralDimension}%
Assume that the IFS
$\Phi \coloneqq \left (\varphi _{i}\right )_{i=1,\ldots n}$ of similarities
is dimensionally regular and let $\varrho $ be the self-similar measure
of $\Phi $ with probability vector
$\left (p_{1},\ldots ,p_{n}\right )\in \left (0,1\right )^{n}$. With
$\zeta $ uniquely determined by
$\tau \left (\zeta \right ) = \zeta $ and
\begin{equation*}
\widetilde{q}\coloneqq \inf \left (\left \{  q>0:-\tau '(q)q+\tau (q)
\leq 1\right \}  \cup \left \{  1\right \}  \right ),
\end{equation*}
the spectral dimension of $\Delta _{\varrho }$ is given by
\begin{equation*}
s_{\varrho }=
\begin{cases}
\zeta  & \text{if\,\,}\zeta \geq \widetilde{q},
\\
{\widetilde{q}}/({1-\tau \left (\widetilde{q}\right )+\widetilde{q}})
& \text{if\,\,}\zeta <\widetilde{q}.
\end{cases}
\end{equation*}
\end{thm} 
\begin{rem}
\label{rem1.6}
We remark in the case $\dim _{S}(\varrho )\geq 1$ we have
$s_{\varrho }=1/2$, and if $\dim _{S}(\varrho )<1$ and
$\dim _{S}(\supp \varrho )=\tau \left (0\right )\leq 1$ then
$s_{\varrho }=\zeta $. Only in the remaining case $s_{\varrho }$ depends
on $\widetilde{q}$; such a case is illustrated in \reftext{Fig.~\ref{fig:Lq_overlap}} on page \pageref{fig:Lq_overlap}.
\end{rem}

Finally, for the more restricted class of $\psi $-Gibbs measures with H\"{o}lder
continuous potential $\psi $ and with respect to a
$\mathcal{C}^{1+\gamma }$-IFS under the OSC, which includes self-similar measures
under the OSC as a special case, we can show that the eigenvalue counting function
obeys a power law with exponent $z_{\varrho }$. For this we use of the
following notation: For
$f,g:\mathbb{R}_{\geq 0}\rightarrow \mathbb{R}_{\geq 0}$ we write
$f\ll g$ if there exists a positive constant $c$ such that
$f(x)\leq cg(x)$ for all $x$ large and we write $f\asymp g$, if
$f\ll g$ and $g\ll f$.
%
%t1.7 #&#
\begin{thm}[Spectral asymptotics] 
\label{thm:SpectralAsymptotic}%
If $\varrho $ is a $\psi$-Gibbs measure for some H\"{o}lder continuous potential
$\psi $ and with respect to a $\mathcal{C}^{1+\gamma }$-IFS satisfying the
OSC, then
\begin{equation*}
N_{\varrho }(x)\asymp x^{z_{\varrho }},
\end{equation*}
where $z_{\varrho }$ is the unique zero of the pressure function as defined
in \reftext{(\ref{eq:Pressure})}.
\end{thm}
 
\section{Dirichlet forms for generalized Kre\u{\i}n--Feller
operators} 
\label{sec:Dirichlet-forms-for} 
\begin{figure}[h]
\center{\begin{tikzpicture}[scale=0.98, every node/.style={transform shape},line cap=round,line join=round,>=triangle 45,x=1cm,y=1cm] \begin{axis}[ x=3.7cm,y=2.5cm, axis lines=middle, axis line style={very thick},ymajorgrids=false, xmajorgrids=false, grid style={thick,densely dotted,black!20}, xlabel= {$q$}, ylabel= {$\beta_\rho (q)$}, xmin=-0.2 , xmax=1.5 , ymin=-0.3, ymax=2.2,x tick style={color=black}, xtick={0,0.428, .4987, 1},xticklabels = {0,$s_\rho\;\,$,$\;\tilde{q}$,1},  ytick={0,1, 2},yticklabels = {0,1,2}] \clip(-0.5,-0.3) rectangle (4,4); 
\draw[line width=.8pt,smooth,samples=180,domain=0.4987:3.4] plot(\x,{log10(0.001^((\x))+0.001^((\x))+0.05^((\x))+0.948^((\x)))/log10(2)}); 
\draw[line width=1pt,smooth,dotted, samples=180,domain=-0.3:0.4987] plot(\x,{log10(0.001^((\x))+0.001^((\x))+0.05^((\x))+0.948^((\x)))/log10(2)}); 
\draw [line width=0.8pt,solid, domain=-0.00 :0.4987] plot(\x,{(1-1.33214*\x)});
\draw [line width=0.7pt,dashed,gray, domain=-0.15 :1.3] plot(\x,{\x});
\node[circle,draw] (c) at (2.48 ,0 ){\,};
 \draw [line width=.7pt,dotted, gray] (0.4987 ,0.)--(0.4987,0.338); 
\draw [line width=.7pt,dotted, gray] (0.428 ,0.)--(0.428,0.43); 
\end{axis} 
\end{tikzpicture}}
\caption{The graph of $\beta _{\varrho }$ (solid line) for a dimensionally
regular IFS with four contraction ratios equal to $1/2$ and an associated
self-similar measure $\varrho $ with probability vector $\left
(0.001,0.001,0.05,0.948\right )$. The graph of $\beta _{\varrho }$ coincides on $\left
[\widetilde{q},1\right ]$ with $\tau $ (dotted line) as defined in \reftext{(\ref{eq:Def_tau})} and we have $\tau \left (0\right )=2$. The linear part
of $\beta _{\varrho }$ is determined by the tangent to the graph of $\tau $ over
the positive $x$-axis through the point $\left (0,1\right )$. The intersection
with the dashed line with slope $1$ gives the value for the spectral dimension
$s_{\varrho }$.}\label{fig:Lq_overlap}
\end{figure}
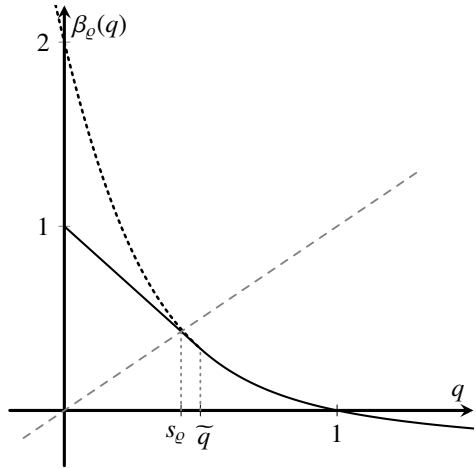

In this section we will define the classical and generalized Kre\u{\i}n--Feller
operator. The spectral properties for the generalized case were studied
in \citep{kuechler1986}, \citep{Volkmer05} and \citep{MR2110540}. The connection
between the generalized and the classical Kre\u{\i}n--Feller operator has
been elaborated in \citep{KSW2019}. In there, it has been shown that the
spectral behavior can be reduced to the classical Kre\u{\i}n--Feller operator
by a straightforward transformation of measure spaces. In the context of
this paper the generalized Kre\u{\i}n--Feller operator will be an important
tool in the proofs of our main results (see e.g. \reftext{Lemma~\ref{lem:ScalingProperty}}). Therefore, in this section we present a short
proof of this fact, which only refers to the Dirichlet form approach.

Our framework closely follows \citep{MR2563669,MR2892328,KN2022}. Throughout
this section, let $\mu $ and $\varrho $ be finite Borel measures on
$\left [a,b\right ]$ such that $\supp (\mu )=[a,b]$, $\mu $ is atomless
and $\varrho \left (\left \{  a,b\right \}  \right )=0$. Let
$L_{\varrho }^{2}=L_{\varrho }^{2}\left (\left [a,b\right ]\right )$ denote
the Hilbert space of square integrable functions with respect to
$\varrho $, $F_{\mu }$ the (strictly increasing and continuous)
\emph{distribution function} of $\mu $ and set
\begin{align*}
H_{\mu }^{1}([a,b]) & \coloneqq \left \{  f:[a,b]\rightarrow \mathbb{R}:
\exists \nabla _{\mu }f\in L_{\mu }^{2}:f(x)=f(a)+\int _{[a,x]}\nabla _{
\mu }f\;\mathrm{d}\mu ,x\in [a,b]\right \}  ,
\\
H_{0,\mu }^{1}([a,b]) & \coloneqq \left \{  f\in H_{\mu }^{1}([a,b]):f(a)=f(b)=0
\right \}
\end{align*}
as well as
\begin{equation*}
\mathcal{C}_{\varrho ,\mu }([a,b])\coloneqq \left \{  f\in \mathcal{C}([a,b]):f\
\text{is aff. lin. in}\ F_{\mu }\ \text{on each comp. of}\ [a,b]
\setminus \supp (\varrho )\right \}  .
\end{equation*}
We say that $f$ is affine linear in $F_{\mu }$ on the interval $I$ if, restricted
to $I$, it can be written as $x\mapsto a+bF_{\mu }$$\left (x\right )$ for
some $a,b\in \mathbb{R}$. Note that $\nabla _{\mu }f$ is unique as an element
of $L_{\mu }^{2}$ (see \citep[Proposition 2.1.3]{Arzt_diss} for a detailed
proof). In the case of the Lebesgue measure $\mu =\Lambda $ we write
$H_{0}^{1}([a,b])\coloneqq H_{0,\Lambda }^{1}([a,b])$ and
$\mathcal{C}_{\varrho }([a,b])\coloneqq \mathcal{C}_{\varrho ,\Lambda }([a,b])$.

Restricted to the (Dirichlet) domain
$\dom \left (\mathcal{E}_{\varrho ,\mu }\right )\times \dom \left (
\mathcal{E}_{\varrho ,\mu }\right )$ with
$\dom (\mathcal{E}_{\varrho ,\mu })\coloneqq H_{0,\mu }^{1}(\left [a,b
\right ])\cap \mathcal{C}_{\varrho ,\mu }([a,b])$ we define the form
\begin{equation*}
\mathcal{E}_{\varrho ,\mu }(f,g)\coloneqq \mathcal{E}_{\varrho ,\mu ,[a,b]}(f,g)
\coloneqq \int _{(a,b)}\nabla _{\mu }f\nabla _{\mu }g\,\;\mathrm{d}\mu ,
\ f,g\in \dom \left (\mathcal{E}_{\varrho }\right ).
\end{equation*}
Again, for the Lebesgue case, we write
$\mathcal{E}_{\varrho }(f,g)\coloneqq \mathcal{E}_{\varrho ,[a,b]}(f,g)
\coloneqq \mathcal{E}_{\varrho ,\Lambda ,[a,b]}(f,g)$. We define the linear
map
\begin{equation*}
\iota _{\mu }:\mathbb{R}^{\left [F_{\mu }(a),F_{\mu }(b)\right ]}\to
\mathbb{R}^{\left [a,b\right ]},\;f\mapsto f\circ F_{\mu },
\end{equation*}
{which} is injective as a consequence of
$F_{\mu }\circ F_{\mu }^{-1}=\idvtex _{\left [F_{\mu }(a),F_{\mu }(b)
\right ]}$. The inverse on its image
\begin{equation*}
\iota _{\mu }^{-1}:\iota _{\mu }\left (\mathbb{R}^{\left [F_{\mu }(a),F_{
\mu }(b)\right ]}\right )\to \mathbb{R}^{\left [F_{\mu }(a),F_{\mu }(b)
\right ]},\;f\mapsto f\circ F_{\mu }^{-1}
\end{equation*}
is therefore bijective and linear.
\begin{lem} 
\label{lem:firstDiffBijection}%
For linear subspaces
$A_{i}\subset \mathbb{R}^{\left [F_{\mu }(a),F_{\mu }(b)\right ]}$,
$i=1,2,3$, with inner product $\left (\cdot ,\cdot \right )_{A_{i}}$ we
consider the restrictions
\begin{equation*}
\iota _{\mu }^{-1}:\iota _{\mu }A_{i}\rightarrow A_{i}
\end{equation*}
and equip $\iota _{\mu }A_{i}$ with the pull-back inner product
$\left (f,g\right )_{\iota _{\mu }A_{i}}\coloneqq \left (\iota _{\mu }^{-1}f,
\iota _{\mu }^{-1}g\right )_{A_{i}}$. This gives rise to the following list

\begin{center}
\begin{tabular}{l|l|l|l|l}
$i$ & $A_{i}$ & $\left (\cdot ,\cdot \right )_{A_{i}}$ & $\iota _{\mu }A_{i}$ & $\left (\cdot ,\cdot \right )_{\iota _{\mu }A_{i}}$ \tabularnewline
\hline
$1$ & $L_{\varrho \circ F_{\mu }^{-1}}^{2}\left (\left [F_{\mu }(a),F_{\mu }(b)\right ]\right )$ & $\left \langle \cdot ,\cdot \right \rangle _{\varrho \circ F_{\mu }^{-1}}$ & $L_{\varrho }^{2}\left (\left [a,b\right ]\right )$ & $\left \langle \cdot ,\cdot \right \rangle _{\varrho }$ \tabularnewline
$2$ & $H_{0}^{1}\left (\left [F_{\mu }(a),F_{\mu }(b)\right ]\right )$ & $\mathcal{E}_{\varrho \circ F_{\mu }^{-1},\left [F_{\mu }(a),F_{\mu }(b)\right ]}$ & $H_{0,\mu }^{1}\left (\left [a,b\right ]\right )$ & $\mathcal{E}_{\varrho ,\mu ,[a,b]}$ \tabularnewline
$3$ & $\dom \left (\mathcal{E}_{\varrho \circ F_{\mu }^{-1},\left [F_{\mu }(a),F_{\mu }(b)\right ]}\right )$ & $\mathcal{E}_{\varrho \circ F_{\mu }^{-1},\left [F_{\mu }(a),F_{\mu }(b)\right ]}$ & $\dom \left (\mathcal{E}_{\varrho ,\mu ,[a,b]}\right )$ & $\mathcal{E}_{\varrho ,\mu ,[a,b]}$ \tabularnewline
\end{tabular}
\end{center}%
and in all cases $\iota _{\mu }^{-1}$ defines an isometric isomorphism of
Hilbert spaces. Moreover, for
$f\in H_{0}^{1}\left (\left [F_{\mu }(a),F_{\mu }(b)\right ]\right )$,
\[
\nabla _{\mu }\left (\iota _{\mu }f\right )=\iota _{\mu }\left (\nabla _{
\Lambda |_{\left [F_{\mu }(a),F_{\mu }(b)\right ]}}f\right )
\]
and for $f\in H_{0,\mu }^{1}\left (\left [a,b\right ]\right )$,
\begin{equation*}
\nabla _{\Lambda |_{\left [F_{\mu }(a),F_{\mu }(b)\right ]}}\left (
\iota _{\mu }^{-1}f\right )=\iota _{\mu }^{-1}\left (\nabla _{\mu }f
\right ).
\end{equation*}
\end{lem}

\begin{proof}
First we show that $\iota _{\mu }A_{i}$ is equal to the claimed spaces.
The case $i=1$ is clear. In treating the cases $i=2,3$ we will also show
the two identities regarding the derivatives. Indeed, for
$f\in H_{0}^{1}\left (\left [F_{\mu }(a),F_{\mu }(b)\right ]\right )$ and
all $x\in [a,b]$, we have
\begin{align*}
f\left (F_{\mu }(x)\right ) & =f\left (F_{\mu }(a)\right )+\int _{
\left [F_{\mu }(a),F_{\mu }(x)\right ]}\nabla _{\Lambda |_{\left [F_{
\mu }(a),F_{\mu }(b)\right ]}}f\;\mathrm{d}\Lambda
\\
& =f\left (F_{\mu }(a)\right )+\int _{\left [F_{\mu }(a),F_{\mu }(x)
\right ]}\nabla _{\Lambda |_{\left [F_{\mu }(a),F_{\mu }(b)\right ]}}f
\;\mathrm{d}\mu \circ F_{\mu }^{-1}
\\
& =f\left (F_{\mu }(a)\right )+\int _{\left [a,x\right ]}\left (
\nabla _{\Lambda |_{\left [F_{\mu }(a),F_{\mu }(b)\right ]}}f\right )
\circ F_{\mu }\;\mathrm{d}\mu .
\end{align*}
Hence,
$f\circ F_{\mu }\in H_{0,\mu }^{1}\left (\left [a,b\right ]\right )$ and
$\nabla _{\mu }\left (f\circ F_{\mu }\right )=\left (\nabla _{\Lambda |_{
\left [F_{\mu }(a),F_{\mu }(b)\right ]}}f\right )\circ F_{\mu }$. Using the
fact that
$f\in \dom \left (\mathcal{E}_{\varrho \circ F_{\mu }^{-1},\left [F_{
\mu }(a),F_{\mu }(b)\right ]}\right )$ is affine linear on the connected
components of
$\left [F_{\mu }(a),F_{\mu }(b)\right ]\setminus \supp \left (\varrho
\circ F_{\mu }^{-1}\right )$, we deduce that $f\circ F_{\mu }$ is affine
linear in $F_{\mu }$ on the components of
$[a,b]\setminus \supp (\varrho )$. Consequently, we have
$f\circ F_{\mu }\in \dom \left (\mathcal{E}_{\varrho ,\mu ,[a,b]}
\right )$. To see the reverse inclusion, note that for
$f\in H_{0,\mu }^{1}\left (\left [a,b\right ]\right )$ and
$x\in \left [F_{\mu }(a),F_{\mu }(b)\right ]$,
\begin{align*}
f\circ F_{\mu }^{-1}\left (x\right ) & =f\left (a\right )+\int _{
\left [a,F_{\mu }^{-1}\left (x\right )\right ]}\nabla _{\mu }f\;
\mathrm{d}\mu
\\
& =f\circ F_{\mu }^{-1}\left (F_{\mu }(a)\right )+\int _{\left [a,F_{
\mu }^{-1}\left (x\right )\right ]}\left (\nabla _{\mu }f\right )\circ F_{
\mu }^{-1}\circ F_{\mu }\;\mathrm{d}\mu
\\
& =f\circ F_{\mu }^{-1}\left (0\right )+\int _{[0,x]}\left (\nabla _{
\mu }f\right )\circ F_{\mu }^{-1}\;\mathrm{d}\Lambda ,
\end{align*}
where we used $F_{\mu }^{-1}\circ F_{\mu }=\idvtex _{[a,b]}$. Since
$a,b\in \supp (\mu )$, it follows
$f\circ F_{\mu }^{-1}\in H_{0}^{1}\left (\left [F_{\mu }(a),F_{\mu }(b)
\right ]\right )$ and
$\nabla _{\Lambda |_{\left [F_{\mu }(a),F_{\mu }(b)\right ]}}\left (f
\circ F_{\mu }^{-1}\right )=\left (\nabla _{\mu }f\right )\circ F_{\mu }^{-1}$.
As above using the fact that
$f\in \dom \left (\mathcal{E}_{\varrho ,\mu ,[a,b]}\right )$ is affine
linear in $F_{\mu }$ on the connected components of
$[a,b]\setminus \supp (\varrho )$, we deduce that
$f\circ F_{\mu }^{-1}$ is affine linear on the components of
$\left [F_{\mu }(a),F_{\mu }(b)\right ]\setminus \supp \left (\varrho
\circ F_{\mu }^{-1}\right )$. Consequently, we have
$f\circ F_{\mu }^{-1}\in \dom \left (\mathcal{E}_{\varrho \circ F_{\mu }^{-1},
\left [F_{\mu }(a),F_{\mu }(b)\right ]}\right )$.

To see that the pull-back inner products are as claimed, we note that the
case $i=1$ is again obvious. For $i=2,3$, we obtain by the above identities
for the derivatives that for $f,g\in H_{0,\mu }^{1}([a,b])$
\begin{align*}
\mathcal{E}_{\varrho ,\mu }\left (f,g\right ) & =\int _{[a,b]}\nabla _{
\mu }\left (f\circ F_{\mu }^{-1}\circ F_{\mu }\right )\nabla _{\mu }
\left (g\circ F_{\mu }^{-1}\circ F_{\mu }\right )\;\mathrm{d}\mu
\\
& =\int _{[a,b]}\left (\nabla _{\Lambda |_{\left [F_{\mu }(a),F_{\mu }(b)
\right ]}}f\circ F_{\mu }^{-1}\right )\circ F_{\mu }\cdot \left (
\nabla _{\Lambda |_{\left [F_{\mu }(a),F_{\mu }(b)\right ]}}g\circ F_{
\mu }^{-1}\right )\circ F_{\mu }\;\mathrm{d}\mu
\\
& =\int _{\left [F_{\mu }(a),F_{\mu }(b)\right ]}\nabla _{\Lambda |_{
\left [F_{\mu }(a),F_{\mu }(b)\right ]}}\left (f\circ F_{\mu }^{-1}
\right )\cdot \nabla _{\Lambda |_{\left [F_{\mu }(a),F_{\mu }(b)\right ]}}
\left (g\circ F_{\mu }^{-1}\right )\;\mathrm{d}\Lambda
\\
& =\mathcal{E}_{\varrho \circ F_{\mu }^{-1},\left [F_{\mu }(a),F_{\mu }(b)
\right ]}\left (f\circ F_{\mu }^{-1},g\circ F_{\mu }^{-1}\right ).
\end{align*}
Since $\iota _{\mu }$ is bijective and $A_{i}$ is a {Hilbert space with} respect to $\left (\cdot ,\cdot \right )_{A_{i}}$ for each for
$i=1,2,3$, we find that $\iota _{\mu }^{-1}$ restricted to
$\iota _{\mu }A_{i}$ defines an isometric isomorphism of Hilbert spaces
in all three cases.
\end{proof}
\begin{prop} 
\label{prop:DenseRichtigGeneralizedKreinFeller}
The set $\dom \left (\mathcal{E}_{\varrho ,\mu }\right )$ is dense in
$L_{\varrho }^{2}$ and equipped with the inner product
\begin{equation*}
(f,g)_{\mathcal{E}_{\varrho ,\mu }}\coloneqq \langle f,g\rangle _{
\varrho }+\mathcal{E}_{\varrho ,\mu }(f,g),
\end{equation*}
defines a Hilbert space, i.e.  $\mathcal{E}_{\varrho ,\mu }$ is closed
with respect to $L_{\varrho }^{2}$.
\end{prop}

\begin{proof}
For the classical case with $\mu =\Lambda $ this follows from
\citep[Proposition 2.2, Proposition 2.3]{KN2022}. The general case follows
from \reftext{Lemma~\ref{lem:firstDiffBijection}} and the fact that for all
$f\in \dom (\mathcal{E}_{\varrho ,\mu })$ we have
$\langle f,g\rangle _{\varrho }\leq \sqrt{\mu ([a,b])\varrho ([a,b])}
\mathcal{E}_{\varrho ,\mu }(f,g)$.
\end{proof}
\begin{rem}
\label{rem2.3}
Since
\begin{equation*}
\langle f,g\rangle _{\varrho }\leq \sqrt{\mu ([a,b])\varrho ([a,b])}
\mathcal{E}_{\varrho ,\mu }(f,g)
\end{equation*}
both bilinear forms $(\cdot ,\cdot )_{\mathcal{E}_{\varrho ,\mu }}$ and
$\mathcal{E}_{\varrho ,\mu }(\cdot ,\cdot )$ give rise to equivalent induced
norms.
\end{rem}

Using \reftext{Proposition~\ref{prop:DenseRichtigGeneralizedKreinFeller}}, we can
define a non-negative, self-adjoint, unbounded operator. Namely, we say
$f\in \dom (\mathcal{E}_{\varrho ,\mu })$ lies in the domain
$\mathcal{D}\left (\Delta _{\varrho ,\mu ,[a,b]}\right )$ of the\emph{ generalized
Kre\u{\i}n--Feller operator}
$\Delta _{\varrho ,\mu ,[a,b]}=\Delta _{\varrho ,\mu }$ if and only if
$g\mapsto \mathcal{E}_{\varrho ,\mu }(g,f)$ extends continuously to a linear
form on $L_{\varrho }^{2}$ and then $\Delta _{\varrho ,\mu }f$ is uniquely
determined by the identity
\begin{align*}
\mathcal{E}_{\varrho ,\mu }(g,f)=\langle g,\Delta _{\varrho ,\mu }f
\rangle _{\varrho },\;\;\text{ for all }\;\;g\in \dom \left (
\mathcal{E}_{\varrho ,\mu }\right ) & .
\end{align*}
If the measure $\mu $ is equal to the Lebesgue measure restricted to
$\left [a,b\right ]$, we call the associated Laplacian
$\Delta _{\varrho }\coloneqq \Delta _{\varrho ,\left [a,b\right ]}
\coloneqq \Delta _{\varrho ,\Lambda ,\left [a,b\right ]}$ the
\emph{classical Kre\u{\i}n--Feller operator}.

An element
$f\in \dom (\mathcal{E}_{\varrho ,\mu })\setminus \left \{  0\right \}  $
is called \emph{eigenfunction }for $\mathcal{E}_{\varrho ,\mu }$ with
\emph{eigenvalue} $\lambda $ if for all
$g\in \dom (\mathcal{E}_{\varrho ,\mu })$, we have
\begin{equation*}
\mathcal{E}_{\varrho ,\mu }(f,g)=\lambda \cdot \langle f,g\rangle _{
\varrho }.
\end{equation*}
The following lemma shows that there is a one-to-one correspondence between
the spectral properties of $\Delta _{\varrho ,\mu ,[a,b]}$ and the spectral
properties of
$\Delta _{\varrho \circ F_{\mu }^{-1},\left [F_{\mu }(a),F_{\mu }(b)
\right ]}$.
\begin{lem} 
\label{lem:_MeasureSpaceTrafo}%
If $f$ is an eigenfunction of $\Delta _{\varrho ,\mu ,[a,b]}$ with eigenvalue
$\lambda $, then $\iota _{\mu }^{-1}\left (f\right )$ is an eigenfunction
with eigenvalue $\lambda $ of
$\Delta _{\varrho \circ F_{\mu }^{-1},\left [F_{\mu }(a),F_{\mu }(b)
\right ]}$. Conversely, if $f$ is an eigenfunction of
$\Delta _{\varrho \circ F_{\mu }^{-1},\left [F_{\mu }(a),F_{\mu }(b)
\right ]}$ with eigenvalue $\lambda $, then
$\iota _{\mu }\left (f\right )$ is an eigenfunction with eigenvalue
$\lambda $ of $\Delta _{\varrho ,\mu ,[a,b]}$.
\end{lem}

\begin{proof}
If $f$ be an eigenfunction of $\Delta _{\varrho ,\mu }$ with eigenvalue
$\lambda $, then for all
$g\in \dom \left (\mathcal{E}_{\varrho ,\mu }\right )$, by \reftext{Lemma~\ref{lem:firstDiffBijection}} and
$F_{\mu }^{-1}\circ F_{\mu }=\idvtex _{[a,b]}$, we have
\begin{align*}
&\int  _{\left [F_{\mu }(a),F_{\mu }(b)\right ]}
\nabla _{\Lambda }\left (f\circ F_{\mu }^{-1}\right )\nabla _{\Lambda }
\left (g\circ F_{\mu }^{-1}\right )\;\mathrm{d}\Lambda\\
 & \quad =\int
 _{\left [F_{\mu }(a),F_{\mu }(b)\right ]}\left (
\nabla _{\mu }f\right )\circ F_{\mu }^{-1}\!\cdot \left (\nabla _{\mu }g
\right )\circ F_{\mu }^{-1}\;\mathrm{d}\mu \circ F_{\mu }^{-1}
\\
&\quad  =\int _{[a,b]}\nabla _{\mu }f\nabla _{\mu }g\;\mathrm{d}\mu =\lambda
\int _{[a,b]}gf\;\mathrm{d}\varrho
\\
&\quad  =\lambda \int  _{\left [F_{\mu }(a),F_{\mu }(b)\right ]}g
\circ F_{\mu }^{-1}\cdot f\circ F_{\mu }^{-1}\;\mathrm{d}\varrho \circ F_{
\mu }^{-1},
\end{align*}
which shows that $f\circ F_{\mu }^{-1}$ is an eigenfunction of
$\Delta _{\varrho \circ F_{\varrho }^{-1},\left [F_{\mu }(a),F_{\mu }(b)
\right ]}$ with eigenvalue $\lambda $. The reverse implication is similar.
\end{proof}
\begin{rem}
\label{rem2.5}
Recall from \citep{KN2022} that the inclusion from the Hilbert space
$\left (\dom \left (\mathcal{E}_{\varrho ,\Lambda }\right ),
\mathcal{E}_{\varrho ,\Lambda }\right )$ into $L_{\varrho }^{2}$ is compact.
Hence, we conclude that there exists an orthonormal system of eigenfunctions
of $\Delta _{\varrho ,\Lambda }$ of $L_{\varrho }^{2}$ with non-negative
eigenvalues
$\left (\lambda _{\varrho ,\Lambda ,[a,b]}^{n}\right )_{n\in
\mathbb{N}}$ in increasing order tending to $\infty $ given
$\supp (\varrho )$ is not finite (see e.g.
\citep[Theorem 4.5.1 and p. 258]{MR1193032}), we write
$\lambda _{\varrho }^{n}\coloneqq \lambda _{\varrho ,\Lambda ,[0,1]}^{n}$. By \reftext{Lemma~\ref{lem:_MeasureSpaceTrafo}} the same holds true for
$\Delta _{\varrho ,\mu ,[a,b]}$ with eigenvalues
$\left (\lambda _{\varrho ,\mu }^{n}\right )_{n\in \mathbb{N}}$.
\end{rem}

We denote the number of eigenvalues of
$\mathcal{E}_{\varrho ,\mu }$ not exceeding $x$ by
$N_{\varrho ,\mu ,[a,b]}\left (x\right )$ and refer to
$N_{\varrho ,\mu ,[a,b]}$ as the
\emph{eigenvalue counting function. }In the case $\mu =\Lambda $ we write
$N_{\varrho ,\Lambda ,[a,b]}=N_{\varrho ,[a,b]}=N_{\varrho }$. The remaining
two observations in this section will play a central role in the proofs
of our main results.
\begin{lem} 
\label{lem:dom_vs_H01_Minmax}%
For all $i\in \mathbb{N}$, we have
\begin{align*}
\lambda _{\varrho ,\mu }^{i} & =\inf \left \{  \sup \left \{
\frac{\mathcal{E}_{\varrho ,\mu }(\psi ,\psi )}{\langle \psi ,\psi \rangle _{\varrho }}:
\psi \in G\setminus \left \{  0\right \}  \right \}  \colon G
\text{ $i$-dim. subspace of }\left (\dom \left (\mathcal{E}_{\varrho ,
\mu }\right ),\mathcal{E}_{\varrho ,\mu }\right )\right \}
\\
& =\inf \left \{  \sup \left \{
\frac{\mathcal{E}_{\varrho ,\mu }(\psi ,\psi )}{\langle \psi ,\psi \rangle _{\varrho }}:
\psi \in G\setminus \left \{  0\right \}  \right \}  \colon G
\text{ $i$-dim. subspace of }\left (H_{0,\mu }^{1}([a,b]),\mathcal{E}_{
\varrho ,\mu }\right )\right \}  .
\end{align*}
\end{lem}

\begin{proof}
This follows from \citep[Lemma 2.7]{KN2022} in tandem with \reftext{Lemma~\ref{lem:_MeasureSpaceTrafo}} and \reftext{Lemma~\ref{lem:firstDiffBijection}}.
\end{proof}
\begin{thm} 
\label{thm:CompareCountingfunctions}
Let $(a_{i})_{i=0,\dots ,n+1}$ be a subdivision vector of $[a,b]$ such
that
\begin{equation*}
a=a_{0}<a_{1}<\cdots <a_{n+1}=b
\end{equation*}
and $\varrho \left (\left \{  a_{i}\right \}  \right )=0$. Then, for all
$x\geq 0$, we have
\begin{equation*}
\sum _{i=0}^{n}N_{\varrho ,\mu ,[a_{i},a_{i+1}]}(x)\leq N_{\varrho ,
\mu }(x)\leq \sum _{i=0}^{n}N_{\varrho ,\mu ,[a_{i},a_{i+1}]}(x)+n.
\end{equation*}
\end{thm}

\begin{proof}
This follows from \citep[Proposition 2.16]{KN2022} and \reftext{Lemma~\ref{lem:_MeasureSpaceTrafo}}.
\end{proof}
 
\section{The $L^{q}$-spectrum}
\label{sec3}

Let us recall some basic facts on the $L^{q}$-spectrum as presented in
\citep{KN2022}.  In this section, let $\varrho $ be any given finite Borel 
measures on $\left [0,1\right ]$ with
$\card \left (\supp (\varrho )\right )=\infty $. Note that
$\lambda $ is eigenvalue of $\Delta _{\varrho }$ if, and only if
$\lambda /\varrho ([0,1])$ is eigenvalue of
$\Delta _{\varrho /\varrho ([0,1])}$. Hence, without loss of generality
we assume that $\varrho $ is a probability measure. We begin this section
with some additional properties of the $L^{q}$-spectrum of
$\varrho $ as given in \reftext{(\ref{eq:Lq-spectrum})}. The function
$\beta _{\varrho }$ will not alter when we take $d$-adic intervals instead
of dyadic ones. (see e.g.
\citep[Proposition 2 and Remarks, p. 466]{MR1312056} or
\citep[Proposition 1.6]{Riedi_diss}) and note that the definition
in \citep[Proposition 1.6]{Riedi_diss} coincides with our definition
for $q\geq 0$. More precisely, for fixed $\delta >0$, let
us define
\[
G_{\delta }\coloneqq \left \{  \left (l\delta ,(l+1)\delta \right ]:l
\in \mathbb{Z},\,\varrho \left (\left (l\delta ,(l+1)\delta \right ]
\right )>0\right \}
\]
and let $(\delta _{n})\in (0,1)^{\mathbb{N}}$ be with
$\delta _{n}\rightarrow 0$ an \emph{admissible} sequence i.e.  there
exists a constant $C>0$ such that for all $n\in \mathbb{N}$ we have
$C\delta _{n}\leq \delta _{n+1}\leq \delta _{n}$. Then, for
$q\geq 0$,
\begin{equation*}
\limsup _{\delta \downarrow 0}\frac{1}{-\log (\delta )}\log \sum _{C
\in G_{\delta }}\varrho \left (C\right )^{q}=\limsup _{m\rightarrow
\infty }\frac{1}{-\log (\delta _{m})}\log \sum _{C\in G_{\delta _{m}}}
\varrho \left (C\right )^{q}.
\end{equation*}
In particular, for $\delta _{m}=2^{-m}$ we obtain
\begin{equation*}
\limsup _{\delta \downarrow 0}\frac{1}{-\log (\delta )}\log \sum _{C
\in G_{\delta }}\varrho \left (C\right )^{q}=\beta _{\varrho }(q).
\end{equation*}
The function $\beta _{\varrho }$ is as a pointwise limit superior of convex
function again convex and we have
\begin{equation*}
\beta _{\varrho }\left (0\right )=\overline{\dim }_{M}\left (\supp
\left (\varrho \right )\right )\,\,\,\,\text{and }\beta _{\varrho }
\left (1\right )=0.
\end{equation*}
Hence, the Legendre transform is given by
\begin{equation*}
\widehat{\beta }_{\varrho }\left (\alpha \right )\coloneqq \inf _{q}
\beta _{\varrho }\left (q\right )+\alpha q.
\end{equation*}
The critical exponent $q_{\varrho }$ defined in the introduction can also
{be characterized as} follows (see \citep[Fact 4.8]{KN2022}).
\begin{align*}
q_{\varrho } & =\inf \left \{  q>0\colon \limsup _{n\to \infty }
\frac{1}{n}\log \sum _{C\in \mathcal{D}_{n}}\left (\varrho \left (C
\right )\Lambda \left (C\right )\right )^{q}\leq 0\right \}
\\
& =\inf \left \{  q>0:\sum _{C\in \mathcal{D}}\left (\varrho \left (C
\right )\Lambda \left (C\right )\right )^{q}<\infty \right \}  =\sup _{
\alpha \geq 0}
\frac{\widehat{\beta }_{\varrho }\left (\alpha \right )}{1+\alpha }.
\end{align*}
 
\section{Iterated function systems and the thermodynamic formalism} 
\label{sec:Conformal-iterated-function}

In the following we consider the special case of weak $\psi $-Gibbs measure
with respect to not necessary linear iterated function systems. For fixed
$n\in \mathbb{N}$ we call the family
$\Phi \coloneqq \left \{  T_{i}:\left [0,1\right ]\to \left [0,1
\right ]\colon i=1,\dots ,n\right \}  $ a
\emph{$\mathcal{C}^{1}$-iterated function systems }($\mathcal{C}^{1}$-IFS)
if its members are $\mathcal{C}^{1}$-maps such that
\begin{enumerate}
\item we have \emph{uniform contraction}, i.e.  for all $j\in I$ we have
$\sup _{x\in [0,1]}\left |T_{j}'(x)\right |<1$,
\item the derivatives $T_{1}',\ldots ,T_{n}'$ are bounded away from zero,
i.e.  for all $i\in I$ we have
$0<\inf _{x\in [0,1]}\left |T_{j}'(x)\right |$,
\item $\Phi $ is non-trivial, i.e.  there is more than one contraction
and the $T_{i}$'s do not share a common fixed point.
\end{enumerate}
If additionally the $T_{1},\dots ,T_{n}$ are
$\mathcal{C}^{1+\gamma }$-maps with $\gamma \in \left (0,1\right )$, we
call the system a
\emph{$\mathcal{C}^{1+\gamma }$ iterated function systems} ($
\mathcal{C}^{1+\gamma }$-IFS). Here $\mathcal{C}^{1+\gamma }$ denotes the
set of differentiable maps with $\gamma $-H\"{o}lder continuous derivative.
We call the unique nonempty compact invariant set
$K\subset \left [0,1\right ]$ of a $\mathcal{C}^{1}$-IFS $\Phi $ the \emph{self-conformal set} associated to $\Phi $.

Let $I\coloneqq \left \{  1,\dots ,n\right \}  $ denote the alphabet and
$I^{m}$ the set of words of length $m\in \mathbb{N}$ over $I$ and by
$I^{*}=\bigcup _{m\in \mathbb{N}}I^{m}\cup \left \{  \varnothing
\right \}  $ we refer to the set of all words with finite length including
the empty word $\varnothing $. Furthermore, the set of words with infinite
length will be denoted by $I^{\mathbb{N}}$ equipped with the metric
$d(x,y)\coloneqq 2^{-\sup \left \{  i\in \mathbb{N}:x_{i}\neq y_{i}
\right \}  }$ and let
$\mathfrak{\mathcal{B}}\left (I^{\mathbb{N}}\right )$ denote the Borel
$\sigma $-algebra of $I^{\mathbb{N}}$. The length of a finite word
$\omega \in I^{*}$ will be denoted by $|\omega |$ and for the concatenation
of $\omega \in I^{*}$ with $x\in I^{*}\cup I^{\mathbb{N}}$ we write
$\omega x$. The shift-map
$\sigma :I^{\mathbb{N}}\cup I^{*}\rightarrow I^{\mathbb{N}}\cup I^{*}$ is
defined by $\sigma (\omega )=\varnothing $ for
$\omega \in I\cup \left \{  \varnothing \right \}  $,
$\sigma (\omega _{1}\cdots \omega _{m})=\omega _{2}\cdots \omega _{m}$
for $\omega _{1}\cdots \omega _{m}\in I^{m}$ with $m>1$ and
$\sigma \left (\omega _{1}\omega _{2},\ldots \right )=\left (\omega _{2}
\omega _{3}\ldots \right )$ for
$\left (\omega _{1}\omega _{2}\cdots \right )\in I^{\mathbb{N}}$. The cylinder set generated by
$\omega \in I^{*}$ is defined by
$\left [\omega \right ]\coloneqq \left \{  \omega x:x\in I^{\mathbb{N}}
\right \}  \subset I^{\mathbb{N}}$. Note that
$\mathfrak{\mathcal{B}}\left (I^{\mathbb{N}}\right )$ is generated by the
set of cylinders sets of arbitrary lengths. The set of $\sigma $-invariant
probability measures on $\mathcal{B}\left (I^{\mathbb{N}}\right )$ is denoted
by $\mathcal{M}_{\sigma }\left (I^{\mathbb{N}}\right )$, where the measure
$\nu $ is called $\sigma $-invariant if
$\nu =\nu \circ \sigma ^{-1}$. Further, for
$u=u_{1}\cdots u_{n}\in I^{n}$, $n\in \mathbb{N}$, we set
$u^{-}=u_{1}\cdots u_{n-1}$. We say $P\subset I^{*}$ is a partition of
$I^{\mathbb{N}}$ if
\begin{equation*}
\bigcup _{\omega \in P}[\omega ]=I^{\mathbb{N}}\:\text{and}\:
[\omega ]\cap [\omega ']=\varnothing ,\,\text{for all }\omega ,
\omega '\in P\:\text{with}\:\omega \neq \omega '.
\end{equation*}
Now, we are able to give a coding of the self-conformal set in terms of
$I^{\mathbb{N}}$. For $\omega \in I^{*}$ we put
$T_{\omega }\coloneqq T_{\omega _{1}}\circ \cdots \circ T_{\omega _{n}}$
and define $T_{\varnothing }\coloneqq \idvtex _{[0,1]}$ to be the identity
map on $[0,1]$. For
$\left (\omega _{1}\omega _{2}\cdots \right )\in I^{\mathbb{N}}$ and
$m\in \mathbb{N}$ we define the initial word by
$\omega |_{n}\coloneqq \omega _{1}\cdots \omega _{n}$. For every
$\omega \in I^{\mathbb{N}}$ the intersection
$\bigcap _{n\in \mathbb{N}}T_{\omega |_{n}}([0,1])$ contains exactly one
point $x_{\omega }\in K$ and gives rise to a surjection
$\pi :I^{\mathbb{N}}\rightarrow K$, $\omega \mapsto x_{\omega }$, which we
call the \emph{natural coding map}. Let
$\mathcal{C}\left (I^{\mathbb{N}}\right )$ denote the space of continuous
real valued functions on $I^{\mathbb{N}}$. Fix
$\psi \in \mathcal{C}\left (I^{\mathbb{N}}\right )$ (sometimes called
\emph{potential function}). For
$f\in \mathcal{C}\left (I^{\mathbb{N}}\right )$ we define the
\emph{Perron-Frobenius operator (with respect to $\psi $)} via
$L_{\psi }f(x)\coloneqq \sum _{y\in \sigma ^{-1}x}\mathrm{e}^{\psi
\left (y\right )}f\left (y\right )$, $x\in I^{\mathbb{N}}$.
\begin{defn} 
\label{def:Gibbs_measure}%
For $f\in \mathcal{C}\left (I^{\mathbb{N}}\right )$,
$\alpha \in (0,1)$ and $n\in \mathbb{N}_{0}$ define
\begin{equation*}
\var _{n}(f)\coloneqq \sup \left \{  \left |f(\omega )-f(u)\right |:
\omega ,u\in I^{\mathbb{N}}\text{ and }\omega _{i}=u_{i}\:
\text{for all } i\in \left \{  1,\dots ,n\right \}  \right \}  ,
\end{equation*}
\begin{equation*}
\left |f\right |_{\alpha }\coloneqq \sup _{n\geq 0}
\frac{\var _{n}(f)}{\alpha ^{n}}\text{ and}\,\:\mathcal{F_{\alpha }}
\coloneqq \left \{  f\in \mathcal{C}\left (I^{\mathbb{N}}\right ):
\left |f\right |_{\alpha }<\infty \right \}  .
\end{equation*}
Elements of $\mathcal{F_{\alpha }}$ are called $\alpha $ \emph{-H\"{o}lder
continuous} functions on $I^{\mathbb{N}}$. Furthermore, the
\emph{Birkhoff sum} of $f$ is defined by
$S_{n}f(x)\coloneqq \sum _{k=0}^{n-1}f\circ \sigma ^{k}(x)$,
$x\in I^{\mathbb{N}}$, $n\in \mathbb{N}$ and $S_{0}f=0$.

For $\psi \in \mathcal{C}\left (I^{\mathbb{N}}\right )$ with
$L_{\psi }\mathbbm{1}=\mathbbm{1}$ let
$\nu \in \mathcal{M}_{\sigma }\left (I^{\mathbb{N}}\right )$ denote a fixed
point probability measure of $L_{\psi }^{*}$, that is
$L_{\psi }^{*}\nu =\nu $ where $L_{\psi }^{*}$ denotes the dual operator
of $L_{\psi }$ acting on the set of Borel probability measures supported
on $I^{\mathbb{N}}$. Such a fixed point always exists by Schauder-Tychonov
fixed point theorem (see also \citep{MR1819804}) and the $\sigma $-invariance
of $\nu $ follows for
$E\in \mathcal{B}\left (I^{\mathbb{N}}\right )$, by
\begin{align*}
\nu (\sigma ^{-1}(E)) & =\int \sum _{j\in I}\mathrm{e}^{S_{n}\psi
\left (j y\right )}\mathbbm{1}_{\sigma ^{-1}(E)}(j y)\,\;\mathrm{d}
\nu (y)=\int \sum _{j\in I}\mathrm{e}^{S_{n}\psi \left (j y\right )}
\mathbbm{1}_{E}(y)\,\;\mathrm{d}\nu (y)=\nu (E).
\end{align*}
We call $\nu $ a \emph{weak $\psi $-Gibbs measure} and
$\varrho \coloneqq \nu \circ \pi ^{-1}$ a\emph{ weak $\psi $-Gibbs measure}
\emph{with respect to the IFS} $\Phi $. For $\omega \in I^{*}$, we define
the measure $\varrho _{\omega }$ and $\Lambda _{\omega }$ by
$\;\mathrm{d}\varrho _{\omega }\coloneqq g_{\omega }\;\mathrm{d}\varrho $ with
$g_{\omega }\coloneqq \mathrm{e}^{S_{\left |\omega \right |}\psi
\circ \pi ^{-1}\circ T_{\omega }}$ and
$\;\mathrm{d}\,\Lambda _{\omega }\coloneqq \left |T'_{\omega }\right |\,
\;\mathrm{d}\Lambda |_{[0,1]}$.
\end{defn}
 
\begin{rem}
\label{rem4.2}
The following list of comments proves useful in our context.
\begin{enumerate}
\item $\nu $ is always a weak Gibbs measure in the sense of
\citep[Proposition 1]{MR1819804}, in particular, for all
$u\in I^{\mathbb{N}}$ and $n\in \mathbb{N}$, we have
\begin{equation}
\mathrm{e}^{-\sum _{i=0}^{n-1}\var _{i}(\psi )}\leq
\frac{\nu \left (\left [u|_{n}\right ]\right )}{\mathrm{e}^{S_{n}\psi (u)}}
\leq \mathrm{e}^{\sum _{i=0}^{n-1}\var _{i}(\psi )}.
\label{eq:WeakGibbsInequality}
\end{equation}
\item The measure $\nu $ has no atoms, since
$\sum _{i=0}^{n}\var _{i}(\psi )=o\left (n\right )$ and
$S_{n}\psi \leq n\max \psi $, where $o$ denotes the usual Landau symbol,
i.e.  $a_{n}=o\left (n\right )$ if $a_{n}/n\to 0$ for
$n\to \infty $.
\item The topological support $\supp (\varrho )$ of $\varrho $ is equal
to $K$. To see this, note that $K$ is covered by the sets
$\bigcup _{\omega \in I^{n}}T_{\omega }\left (\left [0,1\right ]
\right )$, $n\in \mathbb{N}$, and by \reftext{(\ref{eq:WeakGibbsInequality})} each
$T_{\omega }\left (\left [0,1\right ]\right )$ has positive
$\varrho $-measure
$\varrho \left (T_{\omega }\left (\left [0,1\right ]\right )\right )
\geq \exp \left (-\sum _{i=0}^{n-1}\var _{i}(\psi )\right )\nu \left (
\left [\omega \right ]\right )$.
\item If $\psi $ is additionally H\"{o}lder continuous, then $\nu $ is
the unique invariant ergodic $\psi $-Gibbs measure and the bounds in the
above inequality \reftext{(\ref{eq:WeakGibbsInequality})} can be chosen to be positive
constants.
\item For an arbitrary H\"{o}lder continuous function
$\psi :I^{\mathbb{N}}\to \mathbb{R}$ (without assuming
$L_{\psi }\mathbbm{1}=\mathbbm{1}$) there always exists a $\sigma $-invariant
$\psi $-Gibbs measure $\nu $ on the symbolic space as a consequence of
the general thermodynamic formalism and the Perron-Frobenius theorem for
H\"{o}lder potentials (see e.g. \citep{MR2423393}). Let
$h$ denote the only eigenfunction of the Perron-Frobenius operator for
the maximal eigenvalue $\lambda >0$, which is positive and in the same
H\"{o}lder class. Then
$\psi _{1}\coloneqq \psi -\log \lambda +\log h-\log h\circ \sigma $ defines
another H\"{o}lder continuous function for which
$L_{\psi _{1}}\mathbbm{1}=\mathbbm{1}$ and for which $\nu $ is the (unique)
$\psi _{1}$-Gibbs measure, as defined here.
\item If $\psi $ depends only on the first coordinate and is normalized
such that $p_{i}\coloneqq \exp \psi \left (i,\ldots \right )$,
$i\in I$, defines a probability vector, then $\nu $ is in fact a Bernoulli
measure and the bounding constants in the above inequalities (\reftext{(\ref{eq:WeakGibbsInequality})})
can be chosen to be $1$. If additionally the $(T_{i})$ are contracting
similarities, then $\varrho $ coincides with the self-similar measure as
defined in \reftext{(\ref{eq:Simi-1})}.
\end{enumerate}
\end{rem}

Let us define the \emph{geometric potential function}
\begin{align*}
\varphi (\omega _{1}\omega _{2}\cdots )\coloneqq \log \left (\left |T'_{
\omega _{1}}(\pi (\omega _{2}\omega _{3}\cdots ))\right |\right ) & .
\end{align*}
We will make use of the following relation between $\varphi $ with
$T_{\omega }'$ with
$\omega =\omega _{1}\cdots \omega _{n}\in I^{n},n\in \mathbb{N}$. For any
$x\in K$ there exists $\alpha _{x}\in I^{\mathbb{N}}$ such that
$\pi (\alpha _{x})=x$. Hence,
\begin{align*}
\left |T_{\omega }'(x)\right | & =\mathrm{e}^{\sum _{i=1}^{|\omega |}
\log \left (\left |T'_{\omega _{i}}\left (T_{\sigma ^{i}(\omega )}(
\pi (\alpha _{x}))\right )\right |\right )}=\mathrm{e}^{\sum _{i=1}^{|
\omega |}\log \left (\left |T'_{\omega _{i}}\left (\sigma ^{i}(
\omega \pi (\alpha _{x}))\right )\right |\right )}=\mathrm{e}^{S_{n}
\varphi (\omega \alpha _{x})}.
\end{align*}
Note that $\varphi $ is H\"{o}lder continuous if the underlying IFS is
a $\mathcal{C}^{1+\gamma }$-IFS. Moreover, if all the $T_{i}$ are affine,
then $\varphi $ depends only on the first coordinate.

The \emph{pressure }of a continuous function
$f:I^{\mathbb{N}}\to \mathbb{R}$ is defined by
\begin{equation*}
P(f)\coloneqq \lim _{n\to \infty }\frac{1}{n}\log \sum _{\omega \in I^{n}}
\exp \left (S_{\omega }f\right ),
\end{equation*}
with
$S_{\omega }f:=\sup _{x\in \left [\omega \right ]}S_{\left |\omega
\right |}f(x)$. Since for
\begin{equation}
p:t\mapsto P(t\xi )
\label{eq:Pressure}
\end{equation}
with $\xi \coloneqq \psi +\varphi $ is continuous, strictly monotonically
increasing and convex and $\lim _{t\to \pm \infty }p(t)=\mp \infty $, there
exists a unique number $z_{\varrho }\in \mathbb{R}$ such that
$p\left (z_{\varrho }\right )=0$. Moreover, we have
$P\left (\psi \right )=0$.

For $m\in \mathbb{N}$ we will consider the accelerated shift-space
$\left (I^{m}\right )^{\mathbb{N}}$ with natural shift map
$\widetilde{\sigma }:\left (I^{m}\right )^{\mathbb{N}}\to \left (I^{m}
\right )^{\mathbb{N}}$. Clearly, $\left (I^{m}\right )^{\mathbb{N}}$ can
be identified with $I^{\mathbb{N}}$ allowing us to define the
\emph{accelerated ergodic sum} for
$f\in \mathcal{C}\left (I^{\mathbb{N}}\right )$ by
\begin{equation*}
\widetilde{S}_{n}f^{m}(x)\coloneqq\sum_{i=0}^{n-1}f^{m}(\widetilde{\sigma}^{i}(x))\mbox{ with }f^{m}(x)\coloneqq\sum_{i=0}^{m-1}f\left(\sigma^{i}(x)\right).
\end{equation*}
For $\omega \in \left (I^{m}\right )^{*}$ we let $|\omega |_{m}$ denote
the word length of $\omega $ with respect to the alphabet $I^{m}$. With
this setup we have $\widetilde{S}_{n}f^{m}=S_{m\cdot n}f$ and
$\widetilde{S}_{\omega }f^{m}=\sup _{x\in \left [\omega \right ]}
\widetilde{S}_{\left |\omega \right |_{m}}f^{m}(x)$ for
$\omega \in \left (I^{m}\right )^{*}$.
\begin{lem} 
\label{lem:PressureFunctionApprox}%
For $f\in \mathcal{C}\left (I^{\mathbb{N}}\right )$ and every
$m\in \mathbb{N}$, we have
\begin{equation*}
mP\left (f\right )=P_{\widetilde{\sigma }}\left (f^{m}\right )
\coloneqq \lim _{n\to \infty }\frac{1}{n}\log \sum _{\omega \in \left (I^{m}
\right )^{n}}\exp \left (\widetilde{S}_{\omega }f^{m}\right ).
\end{equation*}
\end{lem}

\begin{proof}
The assertion follows immediately from the identity, for
$n\in \mathbb{N}$,
\begin{align*}
\frac{1}{n}\log \sum _{\omega \in \left (I^{m}\right )^{n}}\exp
\left (\widetilde{S}_{\omega }\widetilde{f}\right ) & =\frac{1}{n}
\log \sum _{\omega \in \left (I^{m}\right )^{n}}\exp \left (\sup _{x
\in \left [\omega \right ]}\widetilde{S}_{\left |\omega \right |_{m}}f^{m}(x)
\right )
\\
& =m\frac{1}{mn}\log \sum _{\omega \in I^{mn}}\exp \left (\sup _{x
\in \left [\omega \right ]}S_{\left |\omega \right |}f(x)\right ).\qedhere
\end{align*}
\end{proof}
In the following we show that the
\emph{weak bounded distortion property (wBDP)} holds true for the IFS
$\Phi =\left (T_{1},\dots ,T_{n}\right )$.
\begin{lem}[Weak Bounded Distortion Property] 
\label{lem:weak_Bounded-distortion_Property}
There exists a sequence of non-negative numbers
$\left (b_{m}\right )_{m\in \mathbb{N}}$ with
$b_{m}=o\left (m\right )$ such that for all $\omega \in I^{*}$ and
$x,y\in [0,1]$
\begin{equation*}
\mathrm{e}^{-b_{|\omega |}}\leq \frac{T_{\omega }'(x)}{T_{\omega }'(y)}
\leq \mathrm{e}^{b_{|\omega |}}.
\end{equation*}
\end{lem}

\begin{proof}
Here, we follow the arguments in \citep[Lemma 3.4]{MR2927378}. For
$\omega \coloneqq \omega _{1}\cdots \omega _{l}\in I^{*}$, we have for
all $x,y\in [0,1]$,
\begin{align*}
\frac{T_{\omega }'(x)}{T_{\omega }'(y)} & \leq \exp \left (\sum _{k=1}^{l}
\left |\log \left (\left |T'_{\omega _{k}}\left (T_{\sigma ^{k}\omega }(x)
\right )\right |\right )-\log \left (\left |T'_{\omega _{k}}\left (T_{
\sigma ^{k}\omega }(y)\right )\right |\right )\right |\right )
\\
& \leq \exp \left (\sum _{k=1}^{l}
\underbrace{\max _{x,y\in [0,1]}\max _{i=1,\ldots ,n}\left |\log \left (\left |T'_{i}\left (T_{\sigma ^{k}\omega }(x)\right )\right |\right )-\log \left (\left |T'_{i}\left (T_{\sigma ^{k}\omega }(y)\right )\right |\right )\right |}_{
\eqqcolon A_{l-k}}\right ).
\end{align*}
Let $0<R<1$ be a common bound for the contraction ratios of the maps
$T_{1},\dots ,T_{n}$. Then we have
\begin{equation*}
\left |T_{\sigma ^{k}\omega }(x)-T_{\sigma ^{k}\omega }(y)\right |\leq R^{l-k}|x-y|
\leq R^{l-k}.
\end{equation*}
Hence, we conclude
\begin{align*}
A_{l-k} & \leq \max _{a,b\in \left [0,1\right ],\left |a-b\right |
\leq R^{l-k}}\max _{i=1,\ldots ,n}\left |\log \left (\left |T'_{i}
\left (a\right )\right |\right )-\log \left (\left |T'_{i}\left (b
\right )\right |\right )\right |\eqqcolon B_{l-k}
\end{align*}
Using that each $T'_{1},\dots ,T'_{n}$  is bounded away from zero and
continuous, we obtain $B_{k}\rightarrow 0$ for
$k\rightarrow \infty $. With $b_{m}\coloneqq \sum _{k=0}^{m-1}B_{k}$ we
have $\lim _{m}b_{m}/m$ equals $\lim _{k}B_{k}=0$ as a Ces\`{a}ro limit
and the second inequality holds. The first inequality follows by interchanging
the roles of $x$ and $y$.
\end{proof}
 
\section{Spectral dimensions and asymptotics} 
\label{sec:Spectral-asymptotic}

In this last part we give the proofs of all three main theorems.
 
\subsection{Weak Gibbs measures under
the OSC} 
\label{subsec:The-self-conformal-case}

Let $\varrho $ and $\nu $ be defined as in Section~\ref{sec:Conformal-iterated-function}.
In this section we assume the\emph{ open set condition} (OSC) with feasible
open set $\left (0,1\right )$, i.e.
$T_{i}\left (\left (0,1\right )\right )\cap T_{j}\left (\left (0,1
\right )\right )=\varnothing $ for all $i\neq j$, $i,j\in I$. Note that
in this case $\varrho$ has
no atoms. We start with some basic observations.
\begin{lem} 
\label{lem:ScalingProperty}
For fixed $\omega ,\eta \in I^{*}$ we have that $f$ is an eigenfunction
with eigenvalue $\lambda $ of
$\Delta _{\varrho ,\Lambda ,I_{\omega \eta }}$ with
$I_{\omega \eta }\coloneqq T_{\omega \eta }\left (\left [0,1\right ]
\right )$ if and only if $f\circ T_{\omega }$ is an eigenfunction with eigenvalue
$\lambda $ of
$\Delta _{\varrho _{\omega },\Lambda _{\omega },I_{\eta }}$.
\end{lem}

\begin{proof}
Clearly, by a change of variables
\begin{align*}
\int _{I_{\omega \eta }}(\nabla _{\Lambda |_{I_{\omega \eta }}}f)^{2}\,
\;\mathrm{d}\Lambda & =\int _{I_{\eta }}\left (\left (\nabla _{\Lambda |_{I_{
\omega \eta }}}f\right )\circ T_{\omega }\right ){}^{2}|T_{\omega }'|\,
\;\mathrm{d}\Lambda
\\
& =\int _{I_{\eta }}\nabla _{\Lambda |_{I_{\eta }}}\left (f\circ T_{
\omega }\right )^{2}1/|T_{\omega }'|\,\;\mathrm{d}\Lambda =\int _{I_{
\eta }}\left (\nabla _{\Lambda _{\omega }|_{I_{\eta }}}(f\circ T_{\omega })
\right )^{2}|T_{\omega }'|\,\;\mathrm{d}\Lambda ,
\end{align*}
where we used
$\left (\nabla _{\Lambda |_{I_{\omega \eta }}}f\right )\circ T_{\omega }|T_{
\omega }'|=\nabla _{\Lambda |_{I_{\eta }}}\left (f\circ T_{\omega }
\right )$ and
$\nabla _{\Lambda |_{I_{\eta }}}\left (f\circ T_{\omega }\right )=
\nabla _{\Lambda _{\omega }|_{I_{\eta }}}(f\circ T_{\omega })|T_{\omega }'|$.
For the right hand side of the defining equality of the eigenfunction we
have
\begin{align*}
\int _{I_{\omega \eta }}f^{2}\;\mathrm{d}\varrho = & \int _{[\omega
\eta ]}f^{2}\circ \pi \,\;\mathrm{d}\nu =\int L_{\psi }^{|\omega |}(
\mathbbm{1}_{[\omega \eta ]}(x)f^{2}(\pi (x)))\,\;\mathrm{d}\nu (x)
\\*
& =\int \sum _{j\in I^{|\omega |}}\mathrm{e}^{S_{|\omega |}\psi (j x)}
\mathbbm{1}_{[\omega \eta ]}(j x)f^{2}(\pi (j x))\,\;\mathrm{d}\nu (x)
\\
& =\int \mathrm{e}^{S_{|\omega |}\psi (\omega x)}\mathbbm{1}_{[\eta ]}(x)f^{2}(
\pi (\omega x))\,\;\mathrm{d}\nu (x)=\int _{I_{\eta }}(f\circ T_{\omega })^{2}
\mathrm{e}^{S_{|\omega |}\psi \circ \pi ^{-1}\circ T_{\omega }}\,\;
\mathrm{d}\varrho ,
\end{align*}
where we used the fact that $\pi (\omega x)=T_{\omega }(\pi (x))$.
\end{proof}
Set
$S_{\omega ,\eta }f\coloneqq \sup _{x\in [\omega \eta ]}S_{|\omega |}f(x)$
and
$s_{\omega ,\eta }f\coloneqq \inf _{x\in [\omega \eta ]}S_{|
\omega |}f(x)$. If $\eta $ is the empty word, then
$S_{\omega }f=S_{\omega ,\varnothing }f$ as defined above and we set
$s_{\omega }f\coloneqq s_{\omega ,\varnothing }f$.
\begin{lem}
\label{lem:i-th_Eigenvalue}
For all $i\in \mathbb{N}$ and $\omega ,\eta \in I^{*}$, we have
\begin{equation*}
\frac{\lambda _{\varrho ,\Lambda ,I_{\eta }}^{i}}{\mathrm{e}^{S_{\omega ,\eta }\varphi +b_{|\omega |}+S_{\omega ,\eta }\psi }}
\leq \lambda _{\varrho ,\Lambda ,I_{\omega \eta }}^{i}=\lambda _{
\varrho _{\omega },\Lambda _{\omega },I_{\eta }}^{i}\leq
\frac{\lambda _{\varrho ,\Lambda ,I_{\eta }}^{i}}{\mathrm{e}^{s_{\omega ,\eta }\varphi -b_{|\omega |}+s_{\omega ,\eta }\psi }}
\end{equation*}
where $(b_{m})_{m\in \mathbb{N}}$ is the sequence defined in \reftext{Lemma~\ref{lem:weak_Bounded-distortion_Property}} with
$b_{n}=o\left (n\right )$.
\end{lem}

\begin{proof}
Note that the equality is a direct consequence of \reftext{Lemma~\ref{lem:ScalingProperty}}. For every $f\in H_{0}^{1}(I_{\eta })$ we have
\begin{equation*}
\frac{\int _{I_{\eta }}(\nabla _{\Lambda _{\omega }|_{I_{\eta }}}f)^{2}\;\mathrm{d}\Lambda _{\omega }}{\int _{I_{\eta }}f^{2}\;\mathrm{d}\varrho _{\omega }}=
\frac{\int _{I_{\eta }}(\nabla _{\Lambda |_{I_{\eta }}}f)^{2}\left |T_{\omega }'\right |^{-1}\;\mathrm{d}\Lambda }{\int _{I_{\eta }}f^{2}\mathrm{e}^{S_{|\omega |}\psi \circ \pi ^{-1}\circ T_{\omega }}\;\mathrm{d}\varrho }=
\frac{\int _{I_{\eta }}(\nabla _{\Lambda |_{I_{\eta }}}f)^{2}\left |T_{\omega }'\right |^{-1}\;\mathrm{d}\Lambda }{\int _{I_{\eta }}f^{2}\mathrm{e}^{S_{|\omega |}\psi \circ \pi ^{-1}\circ T_{\omega }}\;\mathrm{d}\varrho }
\end{equation*}
and hence using the wBDP stated \reftext{Lemma~\ref{lem:weak_Bounded-distortion_Property}} gives
\begin{align*}
\frac{1}{\mathrm{e}^{S_{\omega ,\eta }\varphi +b_{|\omega |}+S_{\omega ,\eta }\psi }}
& \leq
\frac{\int _{I_{\eta }}f^{2}\,\;\mathrm{d}\varrho }{\int _{I_{\eta }}(\nabla _{\Lambda |_{I_{\eta }}}f)^{2}\,\;\mathrm{d}\Lambda }
\frac{\int _{I_{\eta }}(\nabla _{\Lambda _{\omega }|_{I_{\eta }}}f)^{2}\;\mathrm{d}\Lambda _{\omega }}{\int _{I_{\eta }}f^{2}\;\mathrm{d}\varrho _{\omega }}
\leq
\frac{1}{\mathrm{e}^{s_{\omega ,\eta }\varphi -b_{|\omega |}+s_{\omega ,\eta }\psi }}.
\end{align*}
Using the fact that
$H_{0}^{1}\left (I_{\eta }\right )=H_{\Lambda _{\omega },0}^{1}\left (I_{
\eta }\right )$ and
$\nabla _{\Lambda _{\omega }|_{I_{\eta }}}f=\nabla _{\Lambda |_{I_{\eta }}}f/
\left |T_{\omega }'\right |$, the claim follows as a consequence of \reftext{Lemma~\ref{lem:dom_vs_H01_Minmax}}.
\end{proof}
\begin{cor} 
\label{cor:Approx}%
For $m\in \mathbb{N}$ large enough, for all
$x>\lambda _{\varrho ,\Lambda }^{1}/r_{m,\min }$, we have
\begin{equation*}
\left (\frac{xr_{m,\min }}{\lambda _{\varrho ,\Lambda }^{1}}\right )^{
\underline{u}_{m}}\leq N_{\varrho ,\Lambda }(x)\leq 2
\frac{x^{\overline{u}_{m}}}{\left (\lambda _{\varrho ,\Lambda }^{1}R_{m,\min }\right )^{\overline{u}_{m}}}+1
\end{equation*}
\textup{where, for} $\omega \in I^{m}$, set $r_{\omega }
\coloneqq \exp (s_{\omega }\varphi -b_{m}+s_{\omega }\psi )$,
$R_{\omega }\coloneqq \exp (S_{\omega }\varphi +b_{m}+S_{\omega }\psi )$,
$r_{m,\min }\coloneqq \min _{i\in I^{m}}r_{i}$,
$R_{m,\min }\coloneqq \min _{i\in I^{m}}R_{i}$ and let
$\underline{u}_{m},\overline{u}_{m}\in \mathbb{R}_{>0}$ be the unique solutions
of
\begin{equation*}
\sum _{\omega \in I^{m}}\mathrm{e}^{\overline{u}_{m}(S_{\omega }
\varphi +S_{\omega }\psi +b_{m})}=\sum _{\omega \in I^{m}}\mathrm{e}^{
\underline{u}_{m}(s_{\omega }\varphi +s_{\omega }\psi -b_{m})}=1.
\end{equation*}
\end{cor}

\begin{proof}
This proof follows the arguments used in \citep[Lemma 2.7]{KL01}. First,
note that for $m\in \mathbb{N}$ sufficiently large for all
$\omega \in I^{m}$ we have
$S_{\omega }\varphi +S_{\omega }\psi +b_{m}<0$ where we used
$b_{m}=o(m)$ and
$S_{\omega }\psi +S_{\omega }\varphi \leq m\left (\max \psi +\max
\varphi \right )$. Therefore there exists
$\overline{u}_{m}\in \mathbb{R}_{>0}$ such that
\begin{equation*}
\sum _{\omega \in I^{m}}R_{\omega }^{\overline{u}_{m}}=1.
\end{equation*}
Moreover, iterating \reftext{Lemma~\ref{lem:i-th_Eigenvalue}} for
$\omega \coloneqq \omega _{1}\cdots \omega _{n}\in \left (I^{m}
\right )^{n}$, $n\in \mathbb{N}$, gives
\begin{equation}
\frac{\lambda _{\varrho ,\Lambda }^{1}}{R_{\omega }}\leq \lambda _{
\varrho ,\Lambda ,I_{\omega }}^{1}\leq
\frac{\lambda _{\varrho ,\Lambda }^{1}}{r_{\omega }}
\label{eq:EigenwerteUngleichung}
\end{equation}
with
$R_{\omega }\coloneqq \prod _{i=1}^{|\omega |_{m}}R_{\omega _{i}}$
and
$r_{\omega }\coloneqq \prod _{i=1}^{|\omega |_{m}}r_{\omega _{i}}$.
Let $x>\lambda _{\varrho , \Lambda }^{1}\eqqcolon \lambda $ be and define
for $m\in \mathbb{N}$ the following partition of
$(I^{m})^{\mathbb{N}}$
\begin{equation*}
P_{m,x}\coloneqq \left \{  \omega \in \left (I^{m}\right )^{*}:R_{
\omega }<\frac{\lambda }{x}\leq R_{\omega ^{-}}\right \}  ,
\end{equation*}
with
$R_{\omega ^{-}}\coloneqq \prod _{i=1}^{|\omega |_{m}-1}R_{\omega _{i}}$. Considering the Bernoulli measure on
$\left (I^{m}\right )^{\mathbb{N}}$ given by the probability vector
$\left (R_{\omega }^{\overline{u}_{m}}\right )$ and using the fact that
$P_{m,x}$ defines a partition of $\left (I^{m}\right )^{\mathbb{N}}$ we
obtain $\sum _{\omega \in P_{m,x}}R_{\omega }^{\overline{u}_{m}}=1$, which
leads to
$\card \left (P_{m.x}\right )\leq x^{\overline{u}_{m}}/\left (
\lambda R_{m,\min }\right )^{\overline{u}_{m}}$. Since for all
$\omega \in P_{m,x}$,
\begin{equation*}
x<\frac{\lambda }{R_{\omega }}\leq \lambda _{\varrho ,\Lambda ,I_{
\omega }}^{1},
\end{equation*}
we conclude from \reftext{Theorem~\ref{thm:CompareCountingfunctions}}
\begin{align*}
N_{\varrho ,\Lambda }(x) & \leq \sum _{\omega \in P _{m,x}}N_{
\varrho ,\Lambda ,I_{\omega }}(x)+2\card \left (P_{m,x}\right )+1=2
\card \left (P_{m,x}\right )+1\\
&\leq 2
\frac{x^{\overline{u}_{m}}}{\left (\lambda R_{m,\min }\right )^{\overline{u}_{m}}}+1.
\end{align*}
For the estimate from below we define for
$x>\frac{\lambda }{r_{m,\min }}$ the following partition of
$(I^{m})^{\mathbb{N}}$
\begin{equation*}
\Xi _{m,x}\coloneqq \left \{  \omega \in \left (I^{m}\right )^{*}:r_{
\omega }<\frac{\lambda }{xr_{m,\min }}\leq r_{\omega ^{-}}\right \}  ,
\end{equation*}
with
$r_{\omega ^{-}}\coloneqq \prod _{i=1}^{|\omega |-1}r_{
\omega _{i}}$. Hence, for all $\omega \in \Xi _{m,x}$, we have by \reftext{(\ref{eq:EigenwerteUngleichung})}
\begin{equation*}
\lambda _{\varrho ,\Lambda ,I_{\omega }}^{1}\leq
\frac{\lambda }{r_{\omega }}\leq
\frac{\lambda }{r_{m,\min }r_{\omega ^{-}}}\leq x.
\end{equation*}
Again, there exists $\overline{u}_{m}\in \mathbb{R}_{>0}$ such that
$\sum _{\omega \in I^{m}}r_{\omega }^{\underline{u}_{m}}=1$ and we obtain
$\sum _{\omega \in \Xi _{m,x}}r_{\omega }^{\underline{u}_{m}}=1$. This implies
\begin{equation*}
1=\sum _{\omega \in \Xi _{m,x}}r_{\omega }^{\underline{u}_{m}}\leq
\left (\frac{\lambda }{xr_{m,\min }}\right )^{\underline{u}_{m}}\card
\left (\Xi _{m,x}\right ),
\end{equation*}
and we conclude from \reftext{Theorem~\ref{thm:CompareCountingfunctions}}
\begin{equation*}
\left (\frac{xr_{m,\min }}{\lambda }\right )^{\underline{u}_{m}}\leq
\card \left (\Xi _{m,x}\right )\leq \sum _{\omega \in \Xi _{m,x}}N_{
\varrho ,\Lambda ,I_{\omega }}(x)\leq N_{\varrho ,\Lambda }(x).\qedhere
\end{equation*}
\end{proof}
In the case of self-similar measures, we obtain the following classical
result of \citep{Fu87}.
\begin{cor} 
\label{cor:LinearCase}%
Assume $0<T_{i}'\equiv \sigma _{i}<1$ and
$\psi (\omega )=\log (p_{\omega _{1}})$, for
$\omega \coloneqq (\omega _{1}\omega _{2}\cdots )\in I^{\mathbb{N}}$, where
$\left (p_{j}\right )_{j}\in \left (0,1\right )^{n}$ is a given probability
vector. Then, for all $i,m\in \mathbb{N}$ and
$\omega =(\omega _{1}\cdots \omega _{m})\in I^{m}$, we have
\begin{equation*}
\lambda _{\varrho ,\Lambda }^{i}=\prod _{j=1}^{m}\sigma _{\omega _{j}}p_{
\omega _{j}}\lambda _{\varrho _{\omega },\Lambda _{\omega }}^{i},
\end{equation*}
and, for all
$x>\lambda _{\varrho ,\Lambda }^{1}\left (\min p_{i}\sigma _{i}\right )^{-1}$,
we have
\begin{equation*}
x^{u}\left (
\frac{\min p_{i}\sigma _{i}}{\lambda _{\varrho ,\Lambda }^{1}}\right )^{u}
\leq N_{\varrho ,\Lambda }(x)\leq
\frac{2x^{u}}{\left (\lambda _{\varrho ,\Lambda }^{1}\min p_{i}\sigma _{i}\right )^{u}}+1,
\end{equation*}
where $u$ is the unique solution of
$\sum {}_{i=1}^{n}(\sigma _{i}p_{i})^{u}=1$.
\end{cor}

The following lemma is elementary and we give its short proof for completeness.
\begin{lem} 
\label{lem:_ConvergenceZeros}%
For $a,b\in \mathbb{R}$ with $a<b$, let
$(f_{n}:\left [a,b\right ]\to \mathbb{R})_{n\in \mathbb{N}}$ be a sequence
of decreasing functions converging pointwise to a function $f$. We assume
that $f_{n}$ has a unique zero in $x_{n}$, for $n\in \mathbb{N}$ and
$f$ has a unique zero in $x$. Then $x=\lim _{n\to \infty }x_{n}$.
\end{lem}

\begin{proof}
Assume that $\lim _{n}x_{n}\neq x$. Then there exists a subsequence
$n_{k}$ such that $x_{n_{k}}\to x^{*}\neq x$ and for all
$k\in \mathbb{N}$ we have
$\left |x-x^{*}\right |/2<\left |x_{n_{k}}-x\right |$ and
$\left |x_{n_{k}}-x^{*}\right |<\left |x-x^{*}\right |/2$. Without loss
of generality we assume $x^{*}<x$. Then $x_{n_{k}}\leq (x^{*}+x)/2$ and
for each $y\in \left ((x^{*}+x)/2,x\right )$, we have
\begin{equation*}
0=f_{n_{k}}(x_{n_{k}})>f_{n_{k}}(y)\geq f_{n_{k}}(x)\to f(x)=0,\:
\text{for }k\to \infty .
\end{equation*}
Consequently, $f(y)=0$ for all $y\in \left ((x^{*}+x)/2,x\right )$, contradicting
the uniqueness of the zero of $f$.
\end{proof}
\begin{lem} 
\label{lem:Approx-1}%
For fixed $m\in \mathbb{N}$ large enough and
$\underline{u}_{m},\overline{u}_{m}\in \mathbb{R}_{>0}$ denoting the unique
solutions of
\begin{equation*}
\sum _{\omega \in I^{m}}\mathrm{e}^{\underline{u}_{m}(S_{\omega }
\varphi +b_{|\omega |}+S_{\omega }\psi )}=\sum _{\omega \in I^{m}}
\mathrm{e}^{\overline{u}_{m}(s_{\omega }\varphi -b_{|\omega |}+s_{
\omega }\psi )}=1,
\end{equation*}
then we have
$\lim _{m\rightarrow \infty }\overline{u}_{m}=\lim _{m\rightarrow
\infty }\underline{u}_{m}=z_{\varrho }$.
\end{lem}

\begin{proof}
Define for $m\in \mathbb{N}$ and $t\geq 0$
\begin{align*}
\underline{P}_{m}(t) & \coloneqq \frac{1}{m}\log
\sum _{\omega \in I^{m}}\exp (t(s_{\omega }\varphi -b_{m}+s_{
\omega }\psi )),\ \\
\overline{P}_{m}(t) & \coloneqq \frac{1}{m}\log
\sum _{\omega \in I^{m}}\exp (t\left (S_{\omega }\varphi +b_{m}+S_{
\omega }\psi \right )) ,
\\
P_{m}(t) & \coloneqq \frac{1}{m}\log
\sum _{\omega \in I^{m}}\exp (tS_{\omega }\xi ).
\end{align*}
We obtain
\begin{align*}
\underline{P}_{m}(t) & \leq P_{m}(t)
\\
& \leq \overline{P}_{m}(t) -t\frac{b_{m}}{m}
\\
& =\frac{1}{m}\log \sum _{\omega \in I^{m}}\exp (t(s_{\omega }
\varphi +s_{\omega }\psi +S_{\omega }\varphi -s_{\omega }\varphi +S_{
\omega }\psi -s_{\omega }\psi ))-t\frac{b_{m}}{m}
\\
& \leq \frac{1}{m}\log \sum _{\omega \in I^{m}}\exp \left (t(s_{
\omega }\varphi +s_{\omega }\psi )+t\left (\sum _{j=0}^{m-1}\var _{j}
\psi +\sum _{j=0}^{m-1}\var _{j}\varphi \right )\right )-t
\frac{b_{m}}{m}
\\
& \leq \underline{P}_{m}(t)+\frac{t}{m}\left (\sum _{j=0}^{m-1}\var _{j}
\varphi +\sum _{j=0}^{m-1}\var _{j}\psi -b_{m}\right ).
\end{align*}
Using the continuity of $\varphi ,\psi $ and
$\lim _{m\rightarrow \infty }b_{m}/m=0$, we deduce
$\lim _{m\rightarrow \infty }\,\overline{P}_{m}(t)=\lim _{m
\rightarrow \infty }\underline{P}_{m}(t)=P(t\xi )$. Furthermore, for all
$t\geq 0$, we have
\begin{align*}
\underline{P}_{m}(t) & \leq \overline{P}_{m}(t)\leq t\frac{b_{m}}{m}+
\frac{1}{m}\log \sum _{\omega \in I^{m}}\exp \left (tm\left (
\max \psi +\max \varphi \right )\right )
\\
& =\log (n)+t\left (\frac{b_{m}}{m}+(\max \psi +\max \varphi )\right ).
\end{align*}
Observe that for $m$ so large that $b_{m}/m\leq -\max \psi /2$, each map
$t\mapsto \overline{P}_{m}(t),t\mapsto \underline{P}_{m}(t)$ and
$t\mapsto P(t)$ is decreasing and has a unique zero in
$\left [0,-\log (n)/\left (\max \psi /2+\max \varphi \right )\right ]$.
Hence the statement follows from \reftext{Lemma~\ref{lem:_ConvergenceZeros}}.
\end{proof}
Now, we can give the proof of \reftext{Theorem~\ref{thm:WeakGibbs}} under  the OSC.
\begin{proof}[Proof of \reftext{Theorem~\ref{thm:WeakGibbs}} under the OSC]
The proof \reftext{Theorem~\ref{thm:WeakGibbs}} assuming the OSC is now an immediate
consequence of \reftext{Corollary~\ref{cor:Approx}} and \reftext{Lemma~\ref{lem:Approx-1}}.
\end{proof}
 
\subsection{Weak Gibbs measures with
overlap} 
\label{subsec:The-self-conformal-case-1}

This section relies on results from
\citep{MR1838304,MR2322179,Barral2020} on the $L^{q}$-spectrum together
with the recent results in \citep{KN2022}. Let $\nu $ and $\varrho $ be
defined as in Section~\ref{sec:Conformal-iterated-function}
and recall that $\Phi $ is non-trivial, i.e.  there is more than one
contraction and the $T_{i}$'s do not share a common fixed point. It is
easy to see that self-similar measures with or without OSC are atomless
as long as $\Phi $ is non-trivial (see \citep{KN2022}). It is an open question
under which condition the same applies to weak Gibbs measures without OSC.
For our purposes it is enough to observe that the non-triviality of
$\Phi $ implies $\card (K)=\infty $ and since $\supp (\varrho )=K$, we
infer the important observation $\varrho ((0,1))>0$. Also note that for
every $\varepsilon >0$ we can extend each $T_{i}$ to an injective contracting
$\mathcal{C}^{1}$-map
$T_{i}:(-\varepsilon ,1+\varepsilon )\rightarrow (-\varepsilon ,1+
\varepsilon )$. Hence, the results of \citep{MR1838304,MR2322179} are valid
in our setting.

First, we will prove that the $L^{q}$-spectrum of $\varrho $ exists in
$(0,1]$. Combining this with \citep[Theorem 1.1, Theorem 1.2]{KN2022} we
conclude that the spectral dimension exists and is given by
$q_{\varrho }$. To this end we need the following lemmata.
\begin{lem} 
\label{lem:WeakGibbsINequality}%
We have for any $G\subset I^{*}$ with
$\biguplus _{u\in G}\left [u\right ]=I^{\mathbb{N}}$ and
$E\in \mathfrak{B}([0,1])$ that
\begin{equation*}
\varrho \left (E\right )\geq \sum _{u\in G}c_{|u|}\nu \left (\left [u
\right ]\right )\varrho \left (T_{u}^{-1}(E)\right )
\end{equation*}
with
$c_{n}\coloneqq \mathrm{e}^{-\sum _{i=0}^{n-1}\var _{i}(\psi )}$ (and therefore
$\log \left (c_{n}\right )=o\left (n\right )$).
\end{lem}

\begin{proof}
For all $E\in \mathfrak{B}([0,1])$ and $u\in I^{*}$, we have
\begin{align*}
\nu \left (\pi ^{-1}(E)\cap [u]\right )= & \int _{[u]}\mathbbm{1}_{E}
\circ \pi \;\mathrm{d}\nu =\int L_{\psi }^{|u|}(\mathbbm{1}_{[u]}(x)
\mathbbm{1}_{E}(\pi (x)))\;\mathrm{d}\nu (x)
\\
= & \int \sum _{j\in I^{|u|}}\mathrm{e}^{S_{|u|}\psi (j x)}
\mathbbm{1}_{[u]}(j x)\mathbbm{1}_{E}(\pi (j x))\;\mathrm{d}\nu (x)
\\
= & \int \mathrm{e}^{S_{|u|}\psi (ux)}\mathbbm{1}_{E}(\pi (ux))\;
\mathrm{d}\nu (x)\geq \mathrm{e}^{-\sum _{i=0}^{|u|-1}\var _{i}(\psi )}
\varrho \left (T_{u}^{-1}\left (E\right )\right )\nu \left (\left [u
\right ]\right ).
\end{align*}
Setting
$c_{n}\coloneqq \mathrm{e}^{-\sum _{i=0}^{n-1}\var _{i}(\psi )}$ and summing
over $u\in G$ we obtain
\begin{equation*}
\varrho \left (E\right )=\sum _{u\in G}\nu \left (\pi ^{-1}(E)\cap [u]
\right )\geq \sum _{u\in G}c_{|u|}\nu \left (\left [u\right ]\right )
\varrho \left (T_{u}^{-1}(E)\right ).
\end{equation*}
Also, the continuity of the potential $\psi $ implies
$\log \left (c_{n}\right )=o\left (n\right )$.
\end{proof}
For $u\in I^{*}$ let us define $K_{u}\coloneqq T_{u}(K)$. Then, for
$n\geq 2$ the set
\begin{equation*}
W_{n}\coloneqq \left \{  u\in I^{*}:\text{diam}\left (K_{u}\right )
\leq 2^{-n}<\text{diam}\left (K_{u^{-}}\right )\right \} ,
\end{equation*}
defines a partition of $I^{\mathbb{N}}$.
 
\begin{lem} 
\label{lem:WeakGibbsEstimation}%
For any $0<q<1$ there exists a sequence
$\left (s_{n}\right )_{n\in \mathbb{N}}\in \mathbb{R}_{>0}^{\mathbb{N}}$
with $\log s_{n}=o\left (n\right )$ such that for every
$n,m\in \mathbb{N}$ and $\widetilde{Q}\in \mathcal{D}_{n}$
\begin{equation*}
\sum _{B\in \mathcal{D}_{n},B\sim \widetilde{Q}}\:\sum _{Q\in
\mathcal{D}_{m+n},Q\subset B}\varrho (Q)^{q}\geq s_{n}\varrho (
\widetilde{Q})^{q}\min _{u\in W_{n}}\sum _{Q\in \mathcal{D}_{m+n}}
\varrho \left (T_{u}^{-1}\left (Q\right )\right )^{q}
\end{equation*}
where $B\sim \widetilde{Q}$ means that the closures of $B$ and
$\widetilde{Q}$ intersect.
\end{lem}

\begin{proof}
As in \citep{MR1838304} for $n,m\in \mathbb{N}$, $u\in W_{n}$ and
$A\in \mathcal{D}_{n}$, let us define
\begin{equation*}
w\left (u,A\right )\coloneqq \sum _{Q\in \mathcal{D}_{n+m}:Q\subset A}
\varrho \left (T_{u}^{-1}\left (Q\right )\right )^{q}.
\end{equation*}
The interval $A\in \mathcal{D}_{n}$ on which $w\left (u,A\right )$ attains
its maximum will be called $q$-heavy for $u\in W_{n}$. We will denote the
$q$-heavy box by $H(u)$ (if there are more than one interval which maximizes
$w(u,\cdot )$ we choose one of them arbitrarily). Note that every
$K_{u}$ with $u\in W_{n}$ intersects at most $3$ intervals in
$\mathcal{D}_{n}$. Hence, we obtain for $u\in W_{n}$,
\begin{align*}
\sum _{Q\in \mathcal{D}_{n+m}}\varrho \left (T_{u}^{-1}\left (Q
\right )\right )^{q}&=\sum _{B'\in \mathcal{D}_{n}}\sum _{Q\in
\mathcal{D}_{n+m}:Q\subset B'}\varrho \left (T_{u}^{-1}\left (Q
\right )\right )^{q}\\
&\leq 3\sum _{Q\in \mathcal{D}_{n+m}:Q\subset H(u)}
\varrho \left (T_{u}^{-1}\left (Q\right )\right )^{q}.
\end{align*}
This leads to
\begin{eqnarray}
\sum _{Q\in \mathcal{D}_{n+m}:Q\subset H(u)}\varrho \left (T_{u}^{-1}
\left (Q\right )\right )^{q}&\geq& \frac{1}{3}\sum _{Q\in \mathcal{D}_{n+m}}
\varrho \left (T_{u}^{-1}\left (Q\right )\right )^{q}\nonumber\\
&\geq& \frac{1}{3}
\min _{v\in W_{n}}\sum _{Q\in \mathcal{D}_{n+m}}\varrho \left (T_{v}^{-1}
\left (Q\right )\right )^{q}.
\label{eq:1/3-Absch=0000E4tzung}
\end{eqnarray}
Further, for every $Q\in \mathcal{D}_{n+m}$ and
$B\in \mathcal{D}_{n}$, by \reftext{Lemma~\ref{lem:WeakGibbsINequality}}, we have
\begin{align*}
\varrho \left (Q\right ) & \geq \sum _{u\in W_{n}}c_{|u|}\nu \left (
\left [u\right ]\right )\varrho \left (T_{u}^{-1}(Q)\right )\geq
\sum _{u\in W_{n}:B=H(u)}c_{|u|}\nu \left (\left [u\right ]\right )
\varrho \left (T_{u}^{-1}(Q)\right )
\\
& \geq \left (\min _{u\in W_{n}}c_{|u|}\right )\sum _{u\in W_{n}:B=H(u)}
\nu \left (\left [u\right ]\right )\varrho \left (T_{u}^{-1}(Q)
\right ).
\end{align*}
Setting
\begin{equation*}
p_{-}\left (B\right )\coloneqq \sum _{u\in W_{n}:B=H(u)}\nu \left (
\left [u\right ]\right )
\end{equation*}
and, if $p_{-}\left (B\right )>0$, using the concavity of the function
$x\mapsto x^{q}$ for $0<q<1$, we obtain
\begin{align*}
\varrho \left (Q\right )^{q} & \geq p_{-}\left (B\right )^{q}\left (
\min _{u\in W_{n}}c_{|u|}\right )^{q}\left (\sum _{u\in W_{n}:B=H(u)}
\frac{\nu \left (\left [u\right ]\right )\varrho \left (T_{u}^{-1}(Q)\right )}{p_{-}\left (B\right )}
\right )^{q}
\\
& \geq p_{-}\left (B\right )^{q-1}\left (\min _{u\in W_{n}}c_{|u|}
\right )^{q}\sum _{u\in W_{n}:B=H(u)}\nu \left (\left [u\right ]
\right )\varrho \left (T_{u}^{-1}(Q)\right )^{q}.
\end{align*}
Summing over $Q\in \mathcal{D}_{n+m}$ with $Q\subset B$, and using \reftext{(\ref{eq:1/3-Absch=0000E4tzung})},
we infer
\begin{align*}
&\sum _{Q\subset B,Q\in \mathcal{D}_{n+m}}\varrho \left (Q\right )^{q} \\
&\quad
\geq p_{-}\left (B\right )^{q-1}\left (\min _{u\in W_{n}}c_{|u|}
\right )^{q}\sum _{u\in W_{n}:B=H(u)}\nu \left (\left [u\right ]
\right )\sum _{Q\subset B,Q\in \mathcal{D}_{n+m}}\varrho \left (T_{u}^{-1}(Q)
\right )^{q}
\\
&\quad  \geq \frac{p_{-}\left (B\right )^{q}}{3}\left (\min _{u\in W_{n}}c_{|u|}
\right )^{q}\min _{v\in W_{n}}\sum _{Q\in \mathcal{D}_{n+m}}\varrho
\left (T_{v}^{-1}\left (Q\right )\right )^{q},
\end{align*}
which is also valid in the case $p_{-}\left (B\right )=0$. For
$\widetilde{Q}\in \mathcal{D}_{n}$ and $u\in W_{n}$ with
$K_{u}\cap \widetilde{Q}\neq \varnothing $ we have
$K_{u}\subset \bigcup _{B\sim \widetilde{Q},B\in \mathcal{D}_{n}}B$, as
a consequence of $\text{diam}\left (K_{u}\right )\leq 2^{-n}$. In particular,
every $K_{u}$ that intersects $\widetilde{Q}$ must have an interval
$B\in \mathcal{D}_{n}$ with $B\sim \widetilde{Q}$ which is $q$-heavy for
$u$. Hence, we obtain
\begin{equation*}
\varrho \left (\widetilde{Q}\right )\leq \sum _{u\in W_{n}:K_{u}\cap
\widetilde{Q}\neq \varnothing }\nu \left (\left [u\right ]\right )
\leq \sum _{B\sim \widetilde{Q}:B\in \mathcal{D}_{n}}\sum _{u\in W_{n}:B=H(u)}
\nu \left (\left [u\right ]\right )=\sum _{B\sim \widetilde{Q}:B\in
\mathcal{D}_{n}}p_{-}\left (B\right ).
\end{equation*}
Using $0<q<1$, we conclude
\begin{equation*}
\varrho \left (\widetilde{Q}\right )^{q}\leq \left (\sum _{B\sim
\widetilde{Q}:B\in \mathcal{D}_{n}}p_{-}\left (B\right )\right )^{q}
\leq \sum _{B\sim \widetilde{Q}:B\in \mathcal{D}_{n}}p_{-}\left (B
\right )^{q}.
\end{equation*}
Summing over all $B\in \mathcal{D}_{n}$ with $B\sim \widetilde{Q}$ gives
\begin{align*}
&\sum _{B\sim \widetilde{Q}:B\in \mathcal{D}_{n}}\sum _{Q\subset B:Q
\in \mathcal{D}_{n+m}}\varrho \left (Q\right )^{q}\\
 &\quad  \ge \sum _{B
\sim \widetilde{Q}:B\in \mathcal{D}_{n}}
\frac{p_{-}\left (B\right )^{q}}{3}\left (\min _{u\in W_{n}}c_{|u|}
\right )^{q}\min _{v\in W_{n}}\sum _{Q\in \mathcal{D}_{n+m}}\varrho
\left (T_{v}^{-1}\left (Q\right )\right )^{q},
\\
&\quad  \geq \frac{\varrho \left (\widetilde{Q}\right )^{q}}{3}\left (\min _{u
\in W_{n}}c_{|u|}\right )^{q}\min _{v\in W_{n}}\sum _{Q\in
\mathcal{D}_{n+m}}\varrho \left (T_{v}^{-1}\left (Q\right )\right )^{q}.
\end{align*}
Note that for every $u\in W_{n}$, by the definition of $W_{n}$, we have
\begin{equation*}
|u|<
\frac{n\log (2)-\log \left (\alpha _{\max }\right )}{-\log \left (\alpha _{\max }\right )},
\end{equation*}
with
$\alpha _{\max }\coloneqq \max _{i=1,\ldots ,n}\max _{x\in [0,1]}
\left |T_{i}'(x)\right |$. Thus, setting
$s_{n}\coloneqq 3^{-1}\min _{u\in W_{n}}c_{|u|}^{q}$ we have
$\lim n^{-1}\log s_{n}=0$, where we used the elementary fact that for any
two sequences
$(x_{n})_{n\in \mathbb{N}}\in \mathbb{R}_{>0}^{\mathbb{N}}$ and
$\left (y_{n}\right )_{n\in \mathbb{N}}\in \mathbb{N}^{\mathbb{N}}$ with
$x_{n}=\text{o}\left (n\right )$, $y_{n}\ll n$, we have
$x_{y_{n}}=\text{o}\left (n\right )$.
\end{proof}
\begin{prop} 
\label{prop:LqSpectrum}%
The $L^{q}$-spectrum $\beta _{\varrho }$ of $\varrho $ exists on
$(0,1]$ as a limit.
\end{prop}

\begin{proof}
Let $0<q<1$. From \citep[Proposition 3.3]{MR2322179} (which holds true
for all Borel probability measures with support $K$, see remark after Proposition
3.3 in \citep{MR2322179}) it follows that there exists a sequence
$\left (b_{q,n}\right )_{n\in \mathbb{N}}$ of positive numbers with
$\log (b_{q,n})=o\left (n\right )$, such that for all
$m,n\in \mathbb{N}$ and $u\in W_{n}$
\begin{equation*}
b_{q,n}\sum _{Q\in \mathcal{D}_{m}}\varrho \left (Q\right )^{q}\leq
\sum _{Q\in \mathcal{D}_{m+n}}\varrho \left (T_{u}^{-1}\left (Q
\right )\right )^{q}.
\end{equation*}
In tandem with \reftext{Lemma~\ref{lem:WeakGibbsEstimation}} we obtain for every
$\widetilde{Q}\in \mathcal{D}_{n}$,
\begin{align*}
\sum _{B\in \mathcal{D}_{n},B\sim \widetilde{Q}}\sum _{Q\in
\mathcal{D}_{m+n}:Q\subset B}\varrho (Q)^{q} & \geq s_{n}\varrho (
\widetilde{Q})^{q}\min _{u\in W_{n}}\sum _{Q\in \mathcal{D}_{m+n}}
\varrho \left (T_{u}^{-1}\left (Q\right )\right )^{q}
\\
& \geq \left (b_{q,n}s_{n}\right )\varrho (\widetilde{Q})^{q}\sum _{Q
\in \mathcal{D}_{m}}\varrho \left (Q\right )^{q}.
\end{align*}
Clearly, $\log \left (b_{q,n}s_{n}\right )=o\left (n\right )$. Hence, we
can apply \citep[Proposition 4.4]{MR2322179}, which shows that
$\beta _{\varrho }$ exists as a limit on $(0,1]$.
\end{proof}
With this knowledge, we can prove the remaining parts of \reftext{Theorem~\ref{thm:WeakGibbs}}.
\begin{proof}[Proof of \reftext{Theorem~\ref{thm:WeakGibbs}} with overlaps]
The proof follows from \reftext{Proposition~\ref{prop:LqSpectrum}} and
\citep[Theorem 1.1, Theorem 1.2]{KN2022} and using
$\varrho ((0,1))>0$, the fact that the $L^{q}$-spectrum of
$\varrho $ and $\varrho |_{\left (0,1\right )}$ coincide on $[0,1]$ as
well as $\beta _{\varrho |_{(0,1)}}$ exists as limit on $(0,1]$.
\end{proof}
The following lemma is needed in the proof of the existence of the Minkowski
dimension for weak Gibbs measures without assuming any separation conditions.
\begin{lem} 
\label{lem:=00005CalphaMIN_endlich}
If $\varrho $ is a weak Gibbs measure for a $\mathcal{C}^{1}$-IFS, then
\begin{equation*}
M_{n}\coloneqq \max \left (-\log \varrho \left (\left (C\right )_{1}
\right ):C\in \mathcal{D}_{n}\right )\ll n,
\end{equation*}
where for $C=\left (2^{-n}k,2^{-n}(k+1)\right ]$,
$k\in \left \{  0,\ldots ,2^{n}-1\right \}  $, we define the centered interval
with triple size as
$\left (C\right )_{1}\coloneqq \left (2^{-n}(k-1),2^{-n}(k+2)\right ]$.
\end{lem}

\begin{proof}
Fix $C\in \mathcal{D}_{n}$ that maximizes
$-\log \varrho \left (\left (C\right )_{1}\right )$. Since
$\varrho \left (C\right )>0$ there exists $u\in W_{n}$ such that
$K_{u}\cap C\neq \varnothing $. Since
$\diam \left (K_{u}\right )\leq 2^{-n}$ we have
$K_{u}\subset \left (C\right )_{1}$ and for arbitrary
$x\in I^{\mathbb{N}}$, the weak Gibbs property gives
$\varrho \left (K_{u}\right )\geq \nu \left (\left [u\right ]\right )
\geq c_{\left |u\right |}\exp \left (S_{\left |u\right |}\psi \left (ux
\right )\right )$ with
$c_{|u|}\coloneqq \mathrm{e}^{-\sum _{i=0}^{|u|-1}\var _{i}\psi }$. Since
$\left |u\right |\leq \left (\log \left (\alpha _{\max }\right )-n
\log (2)\right )/\log \left (\alpha _{\max }\right )$ with
$\alpha _{\max }\coloneqq \max _{i=1,\ldots ,n}\max _{x\in [0,1]}
\left |T_{i}'(x)\right |$ and
$S_{\left |u\right |}\psi \left (ux\right )\geq \left |u\right |\min
\psi $ we get
\begin{equation*}
\varrho \left (\left (C\right )_{1}\right )\geq \varrho \left (K_{u}
\right )\geq c_{\left |u\right |}^{-1}\exp \left (\left |u\right |
\min \psi \right )
\end{equation*}
and further
\begin{align*}
\limsup _{n\rightarrow \infty }
\frac{-\log \varrho \left (\left (C\right )_{1}\right )}{n} & \leq
\limsup _{n\rightarrow \infty }
\frac{\log c_{\left |u\right |}-\min \psi \left |u\right |}{n}
\\
& \leq \limsup _{n\rightarrow \infty }
\frac{\log c_{\left |u\right |}+\min \psi \left (n\log (2)-\log \left (\alpha _{\max }\right )\right )/\log \left (\alpha _{\max }\right )}{n}
\\
& \leq
\frac{\min \psi \log \left (2\right )}{\log \left (\alpha _{\max }\right )}<
\infty ,
\end{align*}
where we used as above the fact that $\log c_{n}=o\left (n\right )$ and
$\min \psi <0$.
\end{proof}
\begin{prop} 
\label{prop:_ApplicationRiedi}%
If $\varrho $ is a weak Gibbs measure for an $\mathcal{C}^{1}$-IFS, then
the upper and lower Minkowski dimension of
$\supp \left (\varrho \right )$ exists, i{.\,e}.
\begin{equation*}
\underline{\dim }_{M}\left (\supp \left (\varrho \right )\right )=
\overline{\dim }_{M}\left (\supp \left (\varrho \right )\right ).
\end{equation*}
In particular, the $L^{q}$-spectrum $\beta _{\varrho }$ exists as a limit
on the closed unit interval.
\end{prop}

\begin{proof}
We will make use of an observation from
\citep[Proposition 2]{MR1312056} that: If we replace
$\varrho \left (C\right )$ with
$\varrho \left (\left (C\right )_{1}\right )$ for
$C\in \mathcal{D}_{n}$, $n\in \mathbb{N}$, in the definition of
$\beta _{\varrho }$, its value does not change for $q\geq 0$. In this way
we can extend $\beta _{\varrho }$ to the negative half-line and denote
this extension, now defined on $\mathbb{R}$, by
$\widetilde{\beta }_{\varrho }$. On the one hand, by \reftext{(\ref{eq:WeakGibbsInequality})},
for all $q\in (0,1]$,
\begin{align*}
\widetilde{\beta }_{\varrho }\left (q\right )&=\beta _{\varrho }\left (q
\right )=\liminf _{n\rightarrow \infty }\frac{1}{\log 2^{n}}\log \sum _{C
\in \mathcal{D}_{n}}\varrho \left (C\right )^{q}\leq \liminf _{n
\rightarrow \infty }
\frac{\log \card \left (\mathcal{D}_{n}\right )}{\log 2^{n}}\\
&=
\underline{\dim }_{M}\left (\supp \left (\varrho \right )\right ).
\end{align*}
Hence,
$\lim _{q\searrow 0}\widetilde{\beta }_{\varrho }\left (q\right )
\leq \underline{\dim }_{M}\left (\supp \left (\varrho \right )\right )$.
On the other hand, for $q<0$, our assumption gives
\begin{align*}
0 & \leq \widetilde{\beta }_{\varrho }\left (q\right )=\limsup _{n
\rightarrow \infty }\frac{1}{\log 2^{n}}\log \sum _{C\in \mathcal{D}_{n}}
\varrho \left (\left (C\right )_{1}\right )^{q}
\\
& \leq \limsup _{n\rightarrow \infty }\frac{1}{\log 2^{n}}\log \left (
\max _{C\in \mathcal{D}_{n}}\varrho \left (\left (C\right )_{1}
\right )^{q}\sum _{C\in \mathcal{D}_{n}}1\right )
\\
& \leq \limsup _{n\rightarrow \infty }
\frac{\log \left (\max _{C\in \mathcal{D}_{n}}\varrho \left (\left (C\right )_{1}\right )^{q}\right )}{n\log 2}+
\frac{\log \card \left (\mathcal{D}_{n}\right )}{\log 2^{n}}
\\
& \leq \limsup _{n\rightarrow \infty }
\frac{-q\max _{C\in \mathcal{D}_{n}}-\log \left (\varrho \left (\left (C\right )_{1}\right )\right )}{n\log 2}+
\frac{\log \card \left (\mathcal{D}_{n}\right )}{\log 2^{n}}
\\
& \leq -q\limsup _{n\rightarrow \infty }\frac{M_{n}}{n\log 2}+
\overline{\dim }_{M}\left (\supp \left (\varrho \right )\right )<
\infty .
\end{align*}
Hence, $\widetilde{\beta }_{\varrho }$ is finite in a neighborhood of
$0$ and in particular continuous in $0$. Consequently, we have
\begin{equation*}
\widetilde{\beta }_{\varrho }\left (0\right )=\lim _{q\searrow 0}
\widetilde{\beta }_{\varrho }\left (q\right )\leq \underline{\dim }_{M}
\left (\supp \left (\varrho \right )\right )\leq \overline{\dim }_{M}
\left (\supp \left (\varrho \right )\right )=\widetilde{\beta }_{
\varrho }\left (0\right ).\qedhere
\end{equation*}
\end{proof}
We finish this section with the proof for the self-similar case without
OSC.
\begin{proof}[Proof of \reftext{Theorem~\ref{thm:SpectralDimension}}]%
In the following we use the results of \citep{Barral2020} to prove \reftext{Theorem~\ref{thm:SpectralDimension}}. We consider contracting similarities that
is, for every $i=1,\dots ,n$, $T_{i}:[0,1]\rightarrow [0,1]$,
$T_{i}(x)=r_{i}x+b_{i}$, $x\in \mathbb{R}$ with
$b_{i}\in \mathbb{R}$ and $\left |r_{i}\right |<1$. For given probability
vector $(p_{1},\ldots ,p_{n})\in (0,1)^{n}$ let $\tau $ be the analytic
function defined in \reftext{(\ref{eq:Def_tau})} and set, as before,
\begin{equation*}
\widetilde{q}\coloneqq \inf \left (\left \{  q\in (0,1):\tau '(q)q-
\tau (q)\geq -1\right \}  \cup \left \{  1\right \}  \right ).
\end{equation*}
Let $\varrho $ be the unique Borel probability measure defined in \reftext{(\ref{eq:Simi-1})}.
By the result from \citep[Theorem 1.2]{Barral2020} the $L^{q}$-spectrum
exists as a limit on $[0,1]$ with
\begin{equation*}
\beta _{\varrho }(q)=
\begin{cases}
1+{\displaystyle
\frac{q\left (\tau (\widetilde{q})-1\right )}{\widetilde{q}}} & q\in [0,
\widetilde{q}),
\\
\tau (q), & q\in [\widetilde{q},1].
\end{cases}
\end{equation*}
\reftext{Fig.~\ref{fig:Lq_overlap}} on page \pageref{fig:Lq_overlap} illustrates
how the spectral dimension depends on the position of
$\widetilde{q}$ in $\left [0,1\right ]$. Now, \reftext{Theorem~\ref{thm:SpectralDimension}} follows from this observation combined with
\citep[Theorem 1.1, Theorem 1.2]{KN2022}.
\end{proof}
 
\subsection{Gibbs measure for $\mathcal{C}^{1+\gamma }$-IFS under the OSC}
\label{sec5.3}

Let $\varrho $ and $\nu $ be defined as in Section~\ref{sec:Conformal-iterated-function}.
In the following we assume that $\psi $ in the definition of the Gibbs
measure $\nu $  is H\"{o}lder continuous and the underlying IFS
$\left \{  T_{1},\dots ,T_{m}\right \}  $ is $\mathcal{C}^{1+\gamma }$, which
implies $\varphi $ is H\"{o}lder continuous, in which case the following
refined bounded distortion property holds (see
\citep[Lemma 3.4]{MR2927378}).
\begin{lem}[Strong Bounded Distortion Property] 
\label{lem:_strong_Bounded-distortion_property}
Assume $T_{1},\dots ,T_{n}$ are $\mathcal{C}^{1+\gamma }$-IFS then we have
the following \emph{strong bounded distortion property} (\textbf{\emph{sBDP}}).
There exists a sequence of positive numbers
$\left (a_{n}\right )_{n\in \mathbb{N}}$ converging to $1$ such that for
$\omega ,\eta \in I^{*}$ and $x,y\in T_{\omega }([0,1])$ we have
\begin{equation*}
a_{|\omega |}^{-1}\leq \frac{T_{\eta }'(x)}{T_{\eta }'(y)}\leq a_{|
\omega |}.
\end{equation*}
\end{lem}

Using the sBDP, we can improve \reftext{Lemma~\ref{lem:i-th_Eigenvalue}} in the following
way.
\begin{lem} 
\label{lem:i-th_Eigenvalue-1}
For all $i\in \mathbb{N}$, $\omega \in I^{*}$ and
$x,y\in I^{\mathbb{N}}$, we have
\begin{equation*}
\frac{\lambda _{\varrho ,\Lambda ,[0,1]}^{i}}{\mathrm{e}^{S_{|\omega |}\varphi (\omega y)+S_{|\omega |}\psi (\omega x)+d_{0}}}
\leq \lambda _{\varrho ,\Lambda ,I_{\omega }}^{i}=\lambda _{\nu _{
\omega },\Lambda _{\omega },[0,1]}^{i}\leq
\frac{\lambda _{\varrho ,\Lambda ,[0,1]}^{i}}{\mathrm{e}^{S_{|\omega |}\varphi (\omega y)+S_{|\omega |}\psi (\omega x)-d_{0}}},
\end{equation*}
where
$d_{0}\coloneqq \log \left (a_{0}\right )+\sum _{k=0}^{^{\infty }}
\var _{k}(\psi )$ and $a_{0}$ is defined in \reftext{Lemma~\ref{lem:_strong_Bounded-distortion_property}}.
\end{lem}

\begin{proof}
For all $\omega \in I^{*}$ and $x,z\in I^{\mathbb{N}}$, we have
\begin{equation*}
\left |S_{|\omega |}\psi (\omega x)-S_{|\omega |}\psi (\omega z)
\right |\leq \sum _{k=0}^{\infty }\var _{k}(\psi )
\end{equation*}
and for all $y,v\in [0,1]$ by \reftext{Lemma~\ref{lem:_strong_Bounded-distortion_property}}, we obtain
\begin{equation*}
\left |\log \left (\left |T_{\omega }'(y)\right |\right )-\log \left (
\left |T_{\omega }'(v)\right |\right )\right |\leq \log \left (a_{0}
\right ).
\end{equation*}
Since there exists $y\in K$ such that $\pi (x)=y$, we obtain
$\log \left (\left |T_{\omega }'(y)\right |\right )=S_{|\omega |}
\varphi (\omega x)$. Thus, we infer
\begin{equation*}
\frac{1}{\mathrm{e}^{S_{|\omega |}\varphi (\omega y)+S_{|\omega |}\psi (\omega x)+d_{0}}}
\leq
\frac{\int _{[0,1]}(\nabla _{\Lambda }f)^{2}\left |T_{\omega }'\right |^{-1}\;\mathrm{d}\Lambda }{\int f^{2}\mathrm{e}^{S_{|\omega |}\psi \circ \pi ^{-1}\circ T_{\omega }}\;\mathrm{d}\varrho }
\leq
\frac{1}{\mathrm{e}^{S_{|\omega |}\varphi (\omega y)+S_{|\omega |}\psi (\omega x)-d_{0}}}.
\end{equation*}
To complete the proof, we can argue in the same way as in the proof of
\reftext{Lemma~\ref{lem:_strong_Bounded-distortion_property}}.
\end{proof}
\begin{lem} 
\label{lem:Partition}%
For every $t>c>0$, we have
\begin{align*}
\Gamma _{t} & \coloneqq \left \{  \omega \in I^{*}\colon S_{\omega }
\xi <\log (c/t)\leq S_{\omega ^{-}}\xi \right \}
\end{align*}
is a partition of $I^{\mathbb{N}}$. In particular, for every
$\omega \in \Gamma _{t}$ and $x\in I^{\mathbb{N}}$, we have
\begin{equation*}
\log (M\mathrm{e}^{d_{0}}t/c)\geq -S_{|\omega |}\xi (\omega x)
\end{equation*}
with $M\coloneqq \exp \left (\max \left (-\xi \right )\right )$ and
$d_{0}\coloneqq \log (a_{0})+\sum _{k=0}^{\infty }\var _{k}(\psi )$ with
$a_{0}$ defined in \reftext{Lemma~\ref{lem:_strong_Bounded-distortion_property}}.
\end{lem}

\begin{proof}
First note, that two cylinder sets are either disjoint or one is contained
in the other. From $\omega \in \Gamma _{t}$ and all $\eta \in I^{*}$, we
have
\begin{align*}
S_{\omega \eta }\xi \leq \sup _{x\in I^{\mathbb{N}}}S_{|\omega |}\xi (
\omega \eta x)\leq & \sup _{x\in I^{\mathbb{N}}}S_{|\omega |}\xi (
\omega x)=S_{\omega }\xi <\log (c/t),
\end{align*}
where we used $\max \xi <0$, which shows for
$\eta \neq \varnothing $ that $\omega \eta \notin \Gamma _{t}$. Moreover,
since $\min \xi <0$, it follows that $S_{\omega }\xi $ converge to
$-\infty $ for $|\omega |\rightarrow \infty $. Consequently, the set
$\Gamma _{t}$ is finite. In particular, for every
$\omega \in I^{\mathbb{N}}$ we have
$S_{\omega |_{n}}\xi \rightarrow -\infty $ as $n$ tends to infinity. Therefore,
there exists $N\in \mathbb{N}$ such that
$S_{\omega |_{N}}\xi <\log (c/t)\leq S_{\omega |_{N-1}}\xi $ and the first
statement follows. For the second claim fix $\omega \in \Gamma _{t}$, then
\begin{align*}
\log (t/c) & \geq -(S_{\omega ^{-}}\xi )=-S_{|\omega -1|}\xi (\omega x)-
\left (S_{\omega ^{-}}\xi -S_{|\omega -1|}\xi (\omega x)\right )
\\
& \geq -S_{|\omega -1|}\xi (\omega x)-d_{0}
\\
& =-S_{|\omega -1|}\xi (\omega x)-\xi \left (\sigma ^{|\omega |-1}
\left (\omega \right )x\right )+\xi \left (\sigma ^{|\omega |-1}
\left (\omega \right )x\right )-d_{0}
\\
& \geq -S_{|\omega |}\xi (\omega x)-\log (M)-d_{0},
\end{align*}
and hence we obtain
$\log \left (M\mathrm{e}^{d_{0}}t/c\right )\geq -S_{|\omega |}\xi (
\omega x)$.
\end{proof}
Recall for $m\in \mathbb{N}$ and
$x\in \left (I^{m}\right )^{\mathbb{N}}$
\begin{equation*}
\xi ^{m}(x)=\sum _{i=0}^{m-1}\xi \left (\sigma ^{i}(x)\right ).
\end{equation*}
 
\begin{lem} 
\label{lem:decompositionDichrichlet}
Set
$d_{0}\coloneqq \log \left (a_{0}\right )+\sum _{k=0}^{\infty }\var _{k}
\psi $ where $a_{0}$ is defined in \reftext{Lemma~\ref{lem:_strong_Bounded-distortion_property}}. Then for $t>c>0$,
$m\in \mathbb{N}$ such that $-m\max \xi -d_{0}>0$,
$x\in \left (I^{m}\right )^{\mathbb{N}}$, we have that
\begin{align*}
\Gamma _{t,m}^{L} & \coloneqq \left \{  \omega \in \left (I^{m}\right )^{*}:-
\widetilde{S}_{|\omega |_{m}}\xi ^{m}(\omega x)\leq \log (t/c)<\min _{v
\in I^{m}}-\widetilde{S}_{|\omega v|_{m}}\xi ^{m}(\omega vx)\right \}
\end{align*}
defines a disjoint family, meaning $\omega \neq \omega '$ implies
$[\omega ]\cap [\omega ']=\varnothing $. With
$k_{m}\coloneqq \exp (-m\max \xi )$ and for every
$\omega \in \left (I^{m}\right )^{*}$
\begin{equation*}
\log \left (t\mathrm{e}^{d_{0}}/(k_{m}c)\right )<-\widetilde{S}_{|
\omega |_{m}}\xi ^{m}(\omega x)\leq \log (t/c),
\end{equation*}
we have $\omega \in \Gamma _{t,m}^{L}$.
\end{lem}

\begin{proof}
For every $\omega \in \Gamma _{t,m}^{L}$ and every $v\in I^{m}$ we have
$\log (t/c)<-\widetilde{S}_{|\omega v|_{m}}\xi ^{m}(\omega vx)$ implying
$\omega v\notin \Gamma _{t,m}^{L}$. Further, using
$-m\max \xi -d_{0}>0$ and the BDP, for every
$\eta \in \left (I^{m}\right )^{*}\setminus \left \{  \varnothing
\right \}  $ we have
\begin{align*}
\log (c/t)&>\widetilde{S}_{|\omega v|_{m}}\xi ^{m}(\omega vx) \\
& \geq
\widetilde{S}_{|\omega v|_{m}}\xi ^{m}(\omega v\eta x)-d_{0}
\\
& =\widetilde{S}_{|\omega v|_{m}}\xi ^{m}(\omega v\eta x)+\sum _{i=0}^{|
\eta |_{m}-1}\left (\xi ^{m}\left (\widetilde{\sigma }^{i}\left (\eta
\right )x\right )-\xi ^{m}\left (\widetilde{\sigma }^{i}\left (\eta
\right )x\right )\right )-d_{0}
\\
& \geq \widetilde{S}_{|\omega v\eta |_{m}}\xi ^{m}(\omega v\eta x)-m
\cdot |\eta |_{m}\max \xi -d_{0}
\\
& \geq \widetilde{S}_{|\omega v\eta |_{m}}\xi ^{m}(\omega v\eta x)-m
\cdot \max \xi -d_{0}
\\
& >\widetilde{S}_{|\omega v\eta |_{m}}\xi ^{m}(\omega v\eta x).
\end{align*}
Thus, for every $\omega \in \Gamma _{t,m}^{L}$ and
$\eta '\in \left (I^{m}\right )^{*}\setminus \left \{  \varnothing
\right \}  $ it follows $\omega \eta '\notin \Gamma _{t,m}^{L}$.

For second assertion fix $x\in \left (I^{m}\right )^{\mathbb{N}}$,
$\omega \in \left (I^{m}\right )^{*}$ and assume
\begin{equation*}
\log \left (t\mathrm{e}^{d_{0}}/(k_{m}c)\right )<-\widetilde{S}_{|
\omega |_{m}}\xi ^{m}(\omega x)\leq \log (t/c).
\end{equation*}
Using the BDP, we obtain for all $v\in I^{m}$,
$\omega \in \left (I^{m}\right )^{*}$,
\begin{align*}
\left |\widetilde{S}_{|\omega |_{m}}\xi ^{m}(\omega x)-\widetilde{S}_{|
\omega |_{m}}\xi ^{m}(\omega vx)\right | & \leq d_{0},
\end{align*}
and consequently,
\begin{align*}
\log (t/c) & <\log (k_{m})-d_{0}-\widetilde{S}_{|\omega |_{m}}\xi ^{m}(
\omega x)
\\
& \leq \log (k_{m})+\xi ^{m}(vx)-\xi ^{m}(vx)-\widetilde{S}_{|\omega |_{m}}
\xi ^{m}(\omega vx)
\\
& \leq -\widetilde{S}_{|\omega v|_{m}}\xi ^{m}(\omega vx).
\end{align*}
Since
$-\widetilde{S}_{|\omega |_{m}}\xi ^{m}(\omega x)\leq \log (t/c)$, we conclude
$\omega \in \Gamma _{t,m}^{L}$.
\end{proof}
Now we are in the position to give the proof of our last main theorem.
\begin{proof}[Proof of \reftext{Theorem~\ref{thm:SpectralAsymptotic}}]
Let $\lambda \coloneqq \lambda _{\varrho ,\Lambda }^{1}$ be the smallest
positive eigenvalue of $\Delta _{\varrho ,\Lambda }$. Then by \reftext{Lemma~\ref{lem:i-th_Eigenvalue-1}} for $\omega \in I^{*}$ and $t>0$ with
$t<\lambda /\exp (S_{\omega }\xi +d_{0})\leq \lambda _{\nu _{\omega },
\Lambda _{\omega }}^{1}$, we have
\begin{equation*}
N_{\varrho ,I_{\omega }}(t)=N_{\varrho _{\omega },\Lambda _{\omega }}(t)=0.
\end{equation*}
Now, by \reftext{Theorem~\ref{thm:CompareCountingfunctions}}, for
$t>c_{R}\coloneqq \lambda \mathrm{e}^{-d_{0}}$, we conclude
\begin{align*}
N_{\varrho }(t)\leq \sum _{\omega \in \Gamma _{t}^{R}}N_{\nu _{\omega },
\Lambda _{\omega }}(t)+2\left |\Gamma _{t}^{R}\right |+1= & 2\left |
\Gamma _{t}^{R}\right |+1,
\end{align*}
where
\begin{align*}
\Gamma _{t}^{R} & =\left \{  \omega \in I^{*}\colon S_{\omega }\xi <
\log (c_{R}/t)\leq S_{\omega ^{-}}\xi \right \}  ,
\end{align*}
which is a partition of $I^{\mathbb{N}}$ by \reftext{Lemma~\ref{lem:Partition}} for
$t>c_{R}$. Hence, for the upper bound, we are left to show that
$\left |\Gamma _{t}^{R}\right |\ll t^{z_{\varrho }}$. For this we use
\citep[Theorem 3.2]{MR3881115} adapted to our situation, i.e.
\begin{equation*}
Z\left (x,t\right )\coloneqq \sum _{n=0}^{\infty }\sum _{\sigma ^{n}y=x}
\mathbbm{1}_{\left \{  -S_{n}\xi (y)\leq \log t\right \}  }\sim G(x,
\log (t))t^{z_{\varrho }},
\end{equation*}
where $\left (x,s\right )\mapsto G(x,s)$, defined on
$I^{\mathbb{N}}\times \mathbb{R}_{>0}$, is bounded from above by inspecting
the corresponding function $G$ in \citep[Theorem 3.2]{MR3881115}, and
$s\mapsto G(x,s)$ is a constant function in the aperiodic case and a periodic
function in the periodic case. Therefore, by \reftext{Lemma~\ref{lem:Partition}}, for  $y\in I^{\mathbb{N}}$,
\begin{equation*}
\card \left (\Gamma _{t}^{R}\right )\leq Z\left (y,\mathrm{e}^{d_{0}+\max (-\xi)}t/
\lambda \right )\ll t^{z_{\varrho }}.
\end{equation*}
For the lower estimate we use an approximation argument involving the strong
bounded distortion property. Applying \reftext{Lemma~\ref{lem:decompositionDichrichlet}} for
$x\in \left (I^{m}\right )^{\mathbb{N}}$ and $m\in \mathbb{N}$ such
$\log (k_{m})-d_{0}>0$ with $k_{m}\coloneqq\exp (-\max \xi ^{m})$, we have
\begin{align*}
\left \{  \omega \in \left (I^{m}\right )^{*}:\log (t\mathrm{e}^{d_{0}}/(k_{m}c))<-
\widetilde{S}_{|\omega |_{m}}\xi ^{m}(\omega x)\leq \log (t/c)\right
\}  & \subset \Gamma _{t,m}^{L}
\end{align*}
and
\begin{equation*}
\Gamma _{t,m}^{L}=\left \{  \omega \in (I^{m})^{*}\colon -
\widetilde{S}_{|\omega |_{m}}\xi ^{m}(\omega x)\leq \log (t/c)<\min _{v
\in I^{m}}-\widetilde{S}_{|\omega v|_{m}}\xi ^{m}(\omega vx)\right \}
\end{equation*}
with $c\coloneqq\mathrm{e}^{d_{0}}\lambda $. By \reftext{Lemma~\ref{lem:i-th_Eigenvalue-1}}, for $\omega \in \Gamma _{t,m}^{L}$, we have
\begin{align*}
\lambda _{\varrho ,\Lambda ,I_{\omega }}^{1} & \leq
\frac{\lambda }{\mathrm{e}^{S_{|\omega |}\xi (\omega x)-d_{0}}}=
\frac{c}{\mathrm{e}^{\widetilde{S}_{|\omega |_{m}}\xi ^{m}(\omega x)-d_{0}}}
\leq t.
\end{align*}
For $t>c$, this leads to
\begin{align*}
N_{\varrho }(t) & \geq \sum _{\omega \in \Gamma _{t,m}^{L}}N_{
\varrho ,I_{\omega }}(t)\geq \card \left (\Gamma _{t,m}^{L}\right )
\\
& \geq \card \left (\left \{  \omega \in \left (I^{m}\right )^{*}:
\log (t\mathrm{e}^{d_{0}}/(k_{m}c))<-\widetilde{S}_{|\omega |_{m}}
\xi ^{m}(\omega x)\leq \log (t/c)\right \}  \right ).
\end{align*}
We conclude
\begin{align*}
N_{\varrho }(t) & \geq \sum _{n=0}^{\infty }\sum _{\omega \in \left (I^{m}
\right )^{n}}\mathbbm{1}_{\left \{  -\widetilde{S}_{|\omega |_{m}}\xi ^{m}(
\omega x)\leq \log (t/c)\right \}  }\\
&\quad -\sum _{n=0}^{\infty }\sum _{
\omega \in \left (I^{m}\right )^{n}}\mathbbm{1}_{\left \{  -
\widetilde{S}_{|\omega |_{m}}\xi ^{m}(\omega x)\leq \log (t\mathrm{e}^{d_{0}}/(k_{m}c))
\right \}  .}
\end{align*}
Moreover, by \reftext{Lemma~\ref{lem:PressureFunctionApprox}} we have
$0=P(z_{\varrho }\xi )=P_{\widetilde{\sigma }}(z_{\varrho }\xi ^{m})$ as
defined in \reftext{Lemma~\ref{lem:PressureFunctionApprox}}. Again,
\citep[Theorem 3.2]{MR3881115} applied to $\xi ^{m}$ gives that there exists
a function $\left (x,s\right )\mapsto \widetilde{G}(x,s)$ defined on
$\left (I^{m}\right )^{\mathbb{N}}\times \mathbb{R}_{>0}$, which is bounded
away from zero by inspecting the corresponding function $G$ in
\citep[Theorem 3.2]{MR3881115}, such that
\begin{align*}
\widetilde{Z}(x,t)\coloneqq \sum _{n=0}^{\infty }\sum _{\omega \in
\left (I^{m}\right )^{n}}\mathbbm{1}_{\left \{  -\widetilde{S}_{|
\omega |_{m}}\xi ^{m}(\omega x)\leq \log (t)\right \}  }\sim
\widetilde{G}(x,\log (t))t^{z_{\varrho }}.
\end{align*}
In the aperiodic case $s\mapsto \widetilde{G}(x,s)$ is a constant function
and hence in this case we immediately get
$t^{z_{\varrho }}\ll N_{\varrho ,\Lambda }(t)$. In the periodic case,
$s\mapsto \widetilde{G}(x,s)$ is periodic with minimal period $a>0$. For
$\ell \coloneqq \left \lceil a/\left (\log (k_{m})-d_{0}\right )
\right \rceil $, we finally have
\begin{align*}
N_{\varrho }(t) & \geq \frac{1}{\ell }\sum _{i=0}^{\ell -1}N_{\varrho }
\left (t\left (\mathrm{e}^{d_{0}}/k_{m}\right )^{i}\right )
\\
& \geq \frac{1}{\ell }\sum _{i=0}^{\ell -1}\card \left (\Gamma _{t
\left (\mathrm{e}^{d_{0}}/k_{m}\right )^{i},m}^{L}\right )
\\
& \geq \frac{1}{\ell }\left (\sum _{i=0}^{\ell -1}\widetilde{Z}\left (x,
\left (\mathrm{e}^{d_{0}}/k_{m}\right )^{i}\left (t/c\right )\right )-
\widetilde{Z}\left (x,\left (\mathrm{e}^{d_{0}}/k_{m}\right )^{i+1}
\left (t/c\right )\right )\right )
\\
& \geq \frac{1}{\ell }\left (\widetilde{Z}\left (x,t/c\right )-
\widetilde{Z}\left (x,\left (\mathrm{e}^{d_{0}}/k_{m}\right )^{\ell }
\left (t/c\right )\right )\right )
\\
& \geq \frac{1}{\ell }\left (\widetilde{Z}\left (x,t/c\right )-
\widetilde{Z}\left (x,\left (\mathrm{e}^{d_{0}}/k_{m}\right )^{a/
\left (\log (k_{m})-d_{0}\right )}\left (t/c\right )\right )\right )
\\
& \sim \frac{1}{\ell }\widetilde{G}\left (x,\log \left (t/c\right )
\right )\left (\frac{t}{c}\right )^{z_{\varrho }}-\widetilde{G}\left (x,
\log \left (t/c\right )\right )\left (\frac{t}{c}\right )^{z_{
\varrho }}\left (\left (\mathrm{e}^{d_{0}}/k_{m}\right )^{a/\left (
\log (k_{m})-d_{0}\right )}\right )^{z_{\varrho }}
\\
& =t^{z_{\varrho }}
\frac{\widetilde{G}\left (x,\log \left (t/c\right )\right )}{c^{z_{\varrho }}\ell }
\left (1-\left (\mathrm{e}^{d_{0}}/k_{m}\right )^{z_{\varrho }a/
\left (\log (k_{m})-d_{0}\right )}\right )
\\
& \gg t^{z_{\varrho }},
\end{align*}
where we used
$\log (t/c)-\log \left (\left (\mathrm{e}^{d_{0}}/k_{m}\right )^{a/
\left (\log (k_{m})-d_{0}\right )}t/c\right )=a$.
\end{proof}
 
\begin{examp}
A natural `geometric' choice for the potential $\psi$ is given by $\delta\varphi$, where $\delta\geq0$ fulfils $P(\delta\varphi)=0$ and, by Bowen's formula,  is equal to the Hausdorff dimension of the self-conformal set $K$. We then have $\xi=(1+\delta)\varphi$,  $P(\delta/(\delta+1)\xi)=0$ and consequently, $s_{\rho}=\delta/(\delta+1)$.
\end{examp}
\section*{Acknowledgement}\addcontentsline{toc}{section}{Acknowledgement}
This research was supported by the {DFG} grant   Ke 1440/3-1. We would like
to thank the anonymous referee for her/his valuable comments, which have
contributed to a significant improvement of the presentation.
 
\phantomsection
\end{document}